\documentclass{tac}

\usepackage{tikz-cd}
\usepackage{enumitem}
\usepackage{mathrsfs}
\usepackage{amsmath}

\usepackage{amssymb}

\usepackage[noabbrev,capitalise,poorman]{cleveref}
\crefformat{equation}{(#2#1#3)}
\crefformat{enumi}{#2#1#3}
\crefformat{enumii}{#2#1#3}
\crefformat{thm}{Theorem #2#1#3}
\crefformat{lem}{Lemma #2#1#3}

\RequirePackage{tikz-cd}
\RequirePackage{amssymb}
\usetikzlibrary{calc}
\usetikzlibrary{decorations.pathmorphing}

\tikzset{curve/.style={settings={#1},to path={(\tikztostart)
			.. controls ($(\tikztostart)!\pv{pos}!(\tikztotarget)!\pv{height}!270:(\tikztotarget)$)
			and ($(\tikztostart)!1-\pv{pos}!(\tikztotarget)!\pv{height}!270:(\tikztotarget)$)
			.. (\tikztotarget)\tikztonodes}},
	settings/.code={\tikzset{quiver/.cd,#1}
		\def\pv##1{\pgfkeysvalueof{/tikz/quiver/##1}}},
	quiver/.cd,pos/.initial=0.35,height/.initial=0}

\tikzset{tail reversed/.code={\pgfsetarrowsstart{tikzcd to}}}
\tikzset{2tail/.code={\pgfsetarrowsstart{Implies[reversed]}}}
\tikzset{2tail reversed/.code={\pgfsetarrowsstart{Implies}}}
\tikzset{no body/.style={/tikz/dash pattern=on 0 off 1mm}}
\usepackage[colorlinks=true]{hyperref}
\hypersetup{allcolors=[rgb]{0.1,0.1,0.4}}

\author{Joshua L. Wrigley}

\thanks{The author is grateful for the support of his supervisor, Olivia Caramello, and acknowledges the financial support of the Insubria-Huawei studentship into ``Grothendieck toposes for information and computation''.}

\address{Dipartimento di scienza e alta tecnologia, Università degli Studi dell'Insubria,\\
 Via Valleggio 11, Como, 22100 (CO), Italy.
}

\title{Some properties of internal locale morphisms externalised}

\copyrightyear{2023}


\eaddress{jwrigley@uninsubria.it}

\newtheorem{thm}{Theorem}
\newtheorem{lem}{Lemma}

\newtheoremrm{rem}{Remark}

\newtheorem{coro}[thm]{Corollary}
\newtheorem{prop}[thm]{Proposition}
\newtheoremrm{df}[thm]{Definition}
\newtheoremrm{dfs}[thm]{Definitions}
\newtheoremrm{ex}[thm]{Example}
\newtheoremrm{exs}[thm]{Examples}
\newtheoremrm{fact}[thm]{Fact}
\newtheoremrm{nota}[thm]{Notation}

\setlist{listparindent = \parindent, parsep=0pt,}
\setenumerate[1]{label = (\roman*), ref = (\roman*)}
\setenumerate[2]{label = (\alph*), ref = (\alph*)}

\mathrmdef{Hom}
\mathbfdef{Set}

\newcommand{\yo}{\!\text{\usefont{U}{min}{m}{n}\symbol{'207}}}

\DeclareFontFamily{U}{min}{}
\DeclareFontShape{U}{min}{m}{n}{<-> udmj30}{}

\newcommand{\lrset}[2]{\left\{\,{#1}\,\middle\vert\,{#2}\,\right\}}
\renewcommand{\to}{\rightarrow}
\newcommand{\Aut}{{\rm Aut}}
\newcommand{\End}{{\rm End}}

\newcommand{\rmopen}{{\rm open}}
\newcommand{\rmcan}{{\rm can}}
\newcommand{\can}{{\rm can}}
\newcommand{\rmtriv}{{\rm triv}}

\newcommand{\BG}{{\bf B}G}
\newcommand{\sets}{{\bf Sets}}
\newcommand{\Finsets}{{\bf FinSets}}

\newcommand{\Loc}{{\bf Loc}}
\newcommand{\Frm}{{\bf Frm}}
\newcommand{\Poset}{{\bf PoSet}}

\newcommand{\CAT}{\mathfrak{CAT}}

\newcommand{\cat}{\mathcal{C}}
\newcommand{\dcat}{\mathcal{D}}

\newcommand{\op}{^{\rm op}}
\newcommand{\id}{{\rm{id}}}
\newcommand{\1}{{\bf 1}}
\newcommand{\2}{{\bf 2}}
\renewcommand{\Hom}{{\rm{Hom}}}
\newcommand{\Sh}{{\bf Sh}}
\newcommand{\opens}{\mathcal{O}}
\newcommand{\Sub}{{\rm{Sub}}}

\newcommand{\Lb}{\mathbb{L}}
\newcommand{\fb}{\mathfrak{f}}

\newcommand{\ftopos}{\mathcal{F}}
\newcommand{\topos}{\mathcal{E}}
\newcommand{\Topos}{{\bf Topos}}
\newcommand{\Geom}{{\bf Geom}}

\newcommand{\CanRelFibr}{\topos \rtimes \ftopos/f^\ast(-)}

\newcommand{\theory}{\mathbb{T}}
\newcommand{\form}[2]{\left(\, \vec{{#2}},  {#1} \,\right)}
\newcommand{\Con}{{\bf Con}_\Sigma}
\newcommand{\Tmod}{\theory\text{-}{\rm mod}}

\begin{document}

\maketitle
\begin{abstract}
 We study morphisms of internal locales of Grothendieck toposes externally: treating internal locales and their morphisms as sheaves and natural transformations.  We characterise those morphisms of internal locales that induce surjective geometric morphisms and geometric embeddings, demonstrating that both can be computed `pointwise'. We also show that the co-frame operations on the co-frame of internal sublocales can also be computed `pointwise' too.
\end{abstract}

\tableofcontents

%%%%%%%%%%%%%%%%%%%%%%%%%%%%%%%%%%%%%%%%%%%%%%%%%%%%%%%%%%%%%%%%

\section{Introduction}\label{sec:intro}

By and large, the topologically interesting data of a space or a continuous map is contained in the algebra of open sets and the inverse image map.  This prompted the shift to `pointfree' topology, as exposited in \cite{pointless}, where \emph{locales} replace spaces and \emph{locale morphisms} replace continuous maps.  Every topos $\topos$ (by topos, we mean Grothendieck topos) comes equipped with a rich internal language, and so also possesses a theory of \emph{internal locales}: those objects that behave, according to the internal language of $\topos$, as a locale.

By choosing certain sites, the internal locales in question correspond to certain structures of interest outside of topos theory.  Therefore, for applications it is beneficial to have a well-developed dictionary externalising notions for internal locales.  Examples of external accounts of internal locale theory can be found in \cite{JT}, \cite[\S C1.6]{elephant} and \cite{fibredsites}.

\subsection{Our contribution}

In this paper, we will show that the most commonly considered properties of internal locale morphisms admit satisfying externalisations.  We study internal locale morphisms in the style of the treatment given for localic toposes and their morphisms in \cite[\S IX]{SGL}.  We will show that 
\begin{enumerate}
	\item surjections of internal locales,
	\item embeddings of internal sublocales,
	\item and the co-frame operations on the co-frame of internal sublocales
\end{enumerate}
can all be computed \emph{`pointwise'}.

%%%%%%%%%%%%%%%%%%%%%%%%%%%%%%%%

\subsection{Overview}  The paper is divided as follows.

\begin{itemize}

\item We begin in \cref{sec:prelim} by recalling some preliminaries.  A brief recount of the theory of (set-based) locales is given in \cref{subsec:prelims:locales}.  We recall in \cref{subsec:prelims:comorph_and_morph} the notions of a comorphism and morphism of sites -- two site-theoretic methods of presenting geometric morphisms.  To accommodate our site-theoretic treatment, we require a result concerning the commutativity of certain diagrams of geometric morphisms induced by a mixture of comorphisms and morphisms of sites.  This is proved in \cref{subsec:prelims:mixer}.

\item In \cref{sec:intloc}, a review is given of the classification of internal locales for the topos $\Sh(\cat,J)$ as established in \cite[Proposition VI.2.2]{JT}, when $\cat$ is assumed to be a cartesian category, and in \cite[Proposition 5.10]{fibredsites} for an arbitrary category $\cat$.  We also review the construction of the relative topos of \emph{internal sheaves} $\Sh(\Lb) \to \topos$ on an internal locale $\Lb$ of $\topos$ as described in \cite[Examples C2.5.8(c)]{elephant} and \cite[Definition 5.2]{fibredsites}.

\item Some examples of internal locales whose base categories are not cartesian are presented in \cref{sec:exs_of_intloc}, including internal locales of toposes of monoid actions.

\item Our study of internal locale morphisms begins in \cref{intlocmorphsec}.  It is well-known (see \cite[\S VI.5]{JT}, \cite[\S 2]{fact1}, or \cite[Corollary 3.5]{fibredsites}) that, given internal locales $\Lb$ and $\Lb'$ of a topos $\topos \simeq \Sh(\cat,J)$, there is an equivalence between internal locale morphisms $\fb \colon \Lb \to \Lb'$ and geometric morphisms $g$ for which the diagram
\[\begin{tikzcd}[column sep = tiny]
	\Sh(\Lb) \ar{rd} \ar{rr}{g} && \Sh(\Lb') \ar{ld} \\
	& \topos &
\end{tikzcd}\]
commutes.  We give an alternate, direct construction of this equivalence, undertaken in \cref{subsec:intmorphgeomorph}.  Our method differs from that of \cite{fibredsites} in that we also demonstrate an equivalence with the morphisms of the underlying sites of the toposes $\Sh(\Lb)$ and $\Sh(\Lb')$ (subject to a commutativity condition).

It is further shown in \cref{subsec:surj}, that the geometric morphism $\Sh(\fb)$ induced by an internal locale morphism $\fb \colon \Lb \to \Lb'$ is \emph{surjective} if and only if each component $\fb^{-1}_c \colon \Lb'(c) \to \Lb(c)$, for $c \in \cat$, is an injective frame homomorphism.  Hence, surjections of internal locales are computed `pointwise'.

\item Those internal locale morphisms that induce embeddings of subtoposes are the subject of \cref{sec:emb}.  We show that \emph{internal locale embeddings} coincide with `pointwise' locale embeddings, i.e. those morphisms of internal locales $\fb \colon \Lb \to \Lb'$ whose component $\fb_c^{-1} \colon \Lb'(c) \to \Lb(c)$ at each $c \in \cat$ is a surjective frame homomorphism.  We also introduce the notion of an \emph{internal nucleus} on an internal locale $\Lb$ of $\Sh(\cat,J)$, and show that these too correspond bijectively with internal sublocale embeddings.

\item We present in \cref{sec:subtoposes_and_quotients} an application of \emph{internal locale embeddings} to categorical logic by giving a clarifying alternative proof of the equivalence between quotient theories of a geometric theory and subtoposes of its classifying topos (see \cite[Examples B4.2.8(i)]{elephant} or \cite[Theorem 3.2.5]{TST}).

\item Finally in \cref{frameintnuc}, we study the co-frame $\Sub_{\Topos}(\Sh(\Lb))$ of subtoposes of $\Sh(\Lb)$ (see \cite[\S A4.5]{elephant} or \cite[\S 4]{TST}).  We show that the co-frame operations of $\Sub_{\Topos}(\Sh(\Lb))$ can be computed `pointwise' via the co-frame operations on $\Sub_{\Loc}(\Lb(c))$, the co-frame of sublocales of $\Lb(c)$, for each $c \in \cat$.

\end{itemize}

%%%%%%%%%%%%%%%%%%%%%%%%%%%%%%%%%%%%%%%%%%%%%%%%%%%%

\section{Preliminaries}\label{sec:prelim}

We begin by recalling some background material.  Familiarity with some topos theory, as can be found in \cite{SGL}, is assumed.  We will denote the bicategory of (Grothendieck) toposes, geometric morphisms, and natural transformations between these, by $\Topos$.

The preliminary material is divided as follows.
\begin{itemize}
	\item Firstly, in \cref{subsec:prelims:locales} we recall some basic notions from locale theory.
	
	\item Our approach to internal locales is site-theoretic.  In \cref{subsec:prelims:comorph_and_morph} we recall the notions of comorphisms and morphisms of sites.
	
	\item Finally, in \cref{subsec:prelims:mixer}, in order to accommodate our site-theoretic approach we prove a necessary result concerning the commutativity of certain diagrams of geometric morphisms induced by both comorphisms and morphisms of sites.
\end{itemize}

\subsection{Frames and locales}\label{subsec:prelims:locales}

If we forget about points, topology is the study of algebras of open sets $\opens(X)$ and the action of continuous maps $f^{-1} \colon \opens(X) \to \opens(Y)$ on these open sets.  The notions of frame and frame homomorphism capture these purely algebraic aspects of topology.

\begin{df}
		A \emph{frame} $L$ is a complete lattice satisfying, for each $\{\,U_i \mid i \in I\,\} \subseteq L$ and $V \in L$, the infinite distributivity law
		\[V \land \bigvee_{i \in I} U_i = \bigvee_{i \in I} V \land U_i.\]
		A \emph{frame homomorphism} is any map between frames that preserves arbitrary joins and finite meets.  We denote the resultant category by $\Frm$.
\end{df}

Our motivating examples, the algebra of opens $\opens(X)$ of a topological space $X$ and the inverse image map $f^{-1} \colon \opens(Y) \to \opens(X)$ of a continuous map $f \colon X \to Y$, are both examples of, respectively, a frame and a frame homomorphism.  To strengthen the analogy with topological spaces, one often works with the category of \emph{locales} $\Loc \simeq \Frm\op$ instead.

\begin{nota}\label{notation}
		For a locale morphism $f \colon L \to K$, we will use $f^{-1} \colon K \to L$ to denote the corresponding frame homomorphism.  Additionally, each frame homomorphism $f^{-1} \colon K \to L$ has a right adjoint $f_\ast \colon L \to K$, since $K$ is complete.
\end{nota}

Frames are equivalently \emph{complete Heyting algebras} (see \cite[Proposition 7.3.2, Appendix 1]{picadopultr}).  The Heyting implication in a frame $L$ is given by
\[
U \to V = \bigvee \lrset{W \in L}{W \land U \leqslant V}.
\]
However, frame homomorphisms need not preserve the Heyting implication.

\begin{df}[\cite{JT}]
	We will say that a frame homomorphism $f \colon L \to K$ is \emph{open} if either of the following equivalent conditions are satisfied:
	\begin{enumerate}
		\item $f \colon L \to K$ is a complete Heyting algebra homomorphism,
		\item $f^{-1} \colon K \to L$ has a left adjoint $\exists_f$ which satisfies the \emph{Frobenius condition}:
		\[
		\exists_f (U \land f^{-1}(V)) = \exists_f (U) \land V,
		\]
		for all $U \in L$ and $V \in K$.
	\end{enumerate}
\end{df}

Open frame homomorphisms generalise open continuous maps (as can be seen by \cite[Proposition IX.7.5]{SGL}).  We will use $\Frm_{\rmopen}$ to denote the category of frames and open frame homomorphisms, and $\Loc_{\rmopen}$ to denote the opposite category $\Frm_{\rmopen}\op$.

%%%%%%%%%%%%%%%%%%%%%%%%%%%%%%%%%%
%%%%%%%%%%%%%%%%%%%%%%%%%%%%%%%%%%%%

\subsection{Comorphisms and morphisms of sites}\label{subsec:prelims:comorph_and_morph}

Here, we recall the notions of comorphisms and morphisms of sites.  These are functors between the underlying categories of two sites that, respectively, covariantly and contravariantly induce geometric morphisms between the respective sheaf toposes.  Both will be central to our site-theoretic approach.  A fuller analysis of both comorphisms and morphisms of sites can be found in \cite{denseness}.

\begin{df}[\cite{SGA4-3}]\label{df:comorphsites}
	Let $(\cat,J)$ and $(\dcat,K)$ be sites.  A \emph{comorphism of sites} 
	\[F \colon (\cat,J) \to (\dcat,K)\]
	is a functor $F \colon \cat \to \dcat$ with the \emph{cover lifting property} -- for each object $c$ of $\cat$ and $K$-covering sieve $S$ on $F(c)$, there exists a $J$-covering sieve $R$ on $c$ such that $F(R) \subseteq S$.
\end{df}

A comorphism of sites $F \colon (\cat,J) \to (\dcat,K)$ yields a geometric morphism 
\[C_F \colon \Sh(\cat,J) \to \Sh(\dcat,K)\]
(see \cite[\S III.2]{SGA4-3}).  In fact, taking the geometric morphism induced by a comorphism of sites is naturally bifunctorial.  Let ${\bf ComorphSites}$ denote the bicategory whose objects are sites, whose 1-cells are comorphisms of sites and whose 2-cells are natural transformations between comorphisms of sites.  There is a bifunctor ${\bf ComorphSites} \to \Topos$ that sends a site to its topos of sheaves, and a comorphism of sites $F$ to its induced geometric morphism $C_F$.

\begin{df}[\cite{SGA4-3}]\label{df:morphsites}
	Let $(\cat,J)$ and $(\dcat,K)$ be sites.  A \emph{morphism of sites} 
	\[F \colon (\cat,J) \to (\dcat,K)\]
	is a functor $F \colon \cat \to \dcat$ satisfying the following conditions.
	\begin{enumerate}
		\item\label{df:morphsites:covpreserv} If $S$ is a $J$-covering sieve on $c \in \cat$, then $F(S)$ is a $K$-covering family of morphisms on $F(c)$.
		\item\label{df:morphsites:terminal} Every object $d$ of $\dcat$ admits a $K$-covering sieve $\{\,d_i \to d \mid i \in I\,\}$ such that each $d_i$, for $i \in I$, has a morphism $d_i \to F(c_i)$ to the image of some $c_i \in \cat$.
		\item\label{df:morphsites:products} For any pair of objects $c_1,\,c_2$ of $\cat$ and any pair of morphisms
		\[g_1 \colon d \to F(c_1), \ g_2 \colon d \to F(c_2)\]
		of $\dcat$, there exists a $K$-covering family \[\{\,h_i \colon d_i \to d \mid i \in I\,\}\] of morphisms in $\dcat$, a pair of families
		\[\{\,f^1_i \colon c_i \to c_1 \mid i \in I\,\}, \ \{\,f^2_i \colon c_i \to c_2\mid i \in I\,\} \]
		of morphisms in $\cat$, and, for each $i \in I$, a morphism $k_i \colon d_i \to F(c'_i)$ such that the squares
		\[\begin{tikzcd}
			d_i \ar{r}{h_i} \ar{d}{k_i} & d \ar{d}{g_1} & d_i \ar{r}{h_i} \ar{d}{k_i} & d \ar{d}{g_2} \\
			F(c_i) \ar{r}{F(f^1_i)} & F(c_1) & F(c_i) \ar{r}{F(f^2_i)} & F(c_2)
		\end{tikzcd}\]
		commute.
		\item\label{df:morphsites:equalizer} For any pair of parallel arrows $f_1, \, f_2 \colon c' \to c$ of $\cat$, and any arrow $g \colon d \to F(c')$ of $\dcat$ such that $F(f_1) \circ g = F(f_2) \circ g$, there exists a $K$-covering family
		\[\{\,h_i \colon d_i \to d\mid i \in I\,\}\] of morphisms of $\dcat$, a family of morphisms
		\[\{\,e_i \colon c_i \to c'\mid i \in I\,\}\]
		of $\cat$ such that $f_1 \circ e_i = f_2 \circ e_i$ for all $ i \in I$, and, for each $i \in I$, a morphism $k_i \colon d_i \to F(c_i)$ such that the square
		\[\begin{tikzcd}
			d_i \ar{r}{h_i} \ar{d}{k_i} & d \ar{d}{g} \\
			F(c_i) \ar{r}{F(e_i)} & F(c')
		\end{tikzcd}\]
		commutes for each $i \in I$.
	\end{enumerate}
\end{df}

\begin{rem}
	In Definition \ref{df:morphsites}, conditions \cref{df:morphsites:terminal} to \cref{df:morphsites:equalizer} express that a functor preserves finite limits \emph{`relatively'}, including those finite limits that do not appear in $\cat$.  Condition \cref{df:morphsites:terminal} expresses that the terminal object is `relatively' preserved, \cref{df:morphsites:products} products, and \cref{df:morphsites:equalizer} equalizers.  If $\cat$ and $\dcat$ are both cartesian categories, then a functor $F \colon \cat \to \dcat$ satisfies conditions \cref{df:morphsites:terminal} to \cref{df:morphsites:equalizer} if and only if $F$ preserves finite limits.
\end{rem}

A morphism of sites $F \colon (\cat,J) \to (\dcat,K)$ induces a geometric morphism
\[
\Sh(F) \colon \Sh(\dcat, K) \to \Sh(\cat,J)
\]
(see \cite[\S III.1]{SGA4-3}).  Just as with comorphisms of sites, taking the geometric morphism induced by a morphism of sites is naturally bifunctorial.  Let ${\bf MorphSites}$ denote the bicategory whose objects are sites, whose 1-cells are morphisms of sites and whose 2-cells are natural transformations between morphisms of sites.  There exists a bifunctor ${\bf MorphSites}\op \to \Topos$ that sends a site to its topos of sheaves and a morphism of sites to its induced geometric morphism.

%%%%%%%%%%%%%%%%%%%%%%%%%%%%%%%
%%%%%%%%%%%%%%%%%%%%%%%%%%%%%%%%

\subsection{Mixing comorphisms and morphisms of sites}\label{subsec:prelims:mixer}

To accommodate the site-theoretic approach we will take, we require a result concerning the commutativity of certain diagrams induced by geometric morphisms involving both comorphisms and morphisms of sites.

Although we will only encounter (strict) functors and Grothendieck fibrations in this paper, the following result most naturally exists in the language of \emph{Street fibrations}, introduced in \cite{streetfibr1} and developed further in \cite{streetfibr2} -- a weakening of the notion of Grothendieck fibration that accommodates the `principle of equivalence' (that constructions in category theory should only be defined up to equivalence and not equality).  Just as cloven Grothendieck fibrations correspond to functors $P \colon \cat\op \to \CAT$ (see \cite[\S B1.3]{elephant}), cloven Street fibrations correspond to \emph{pseudo-functors} $P \colon \cat\op \to \CAT$ (see \cite[\S 2.2]{relsites}).

%%%%%%%%%%%%%%%%%%%%%%%%%%%%%%%%%%%%%%%%%%%%%%%%%%

A \emph{Street fibration} is a functor $A \colon \cat \to \topos$ such that for each object $c \in \cat$ and an arrow $e \xrightarrow{f} A(c)$, there exists a (weak) cartesian lifting $d \xrightarrow{g} c$ of $f$ by which we mean that there exists a distinguished isomorphism $h \colon e \xrightarrow{\sim} P(d)$ such that $P(g) \circ h = f$ and which is cartesian in the sense that, for any arrows $d' \xrightarrow{g'} c \in \cat$ and $A(d') \xrightarrow{k} A(d) \in \cat$ for which the triangle
\[\begin{tikzcd}
	A(d') \ar{r}{k} \ar{rd}[']{A(g')} & A(d) \ar{d}{A(g)} \\
	& A(c)
\end{tikzcd}\]
commutes, there exists a unique arrow $d' \xrightarrow{k'} d$ of $\cat$ such that the triangle
\[
\begin{tikzcd}
	d' \ar{r}{k'} \ar{rd}[']{g'} & d \ar{d}{g} \\
	& c 
\end{tikzcd}
\]
commutes and $A(k') =k$.

A \emph{morphism of Street fibrations} $A \colon \cat \to \topos$ and $B \colon \dcat \to \ftopos$ is defined as a pair of functors $F \colon \cat \to \dcat$ and $G \colon \topos \to \ftopos$ such that $F$ sends cartesian arrows to cartesian arrows and the square
\[
\begin{tikzcd}
	\cat \ar{r}{F} \ar{d}{A} \ar[phantom]{rd}[description]{\cong} & \dcat \ar{d}{B} \\
	\topos \ar{r}{G} & \ftopos
\end{tikzcd}
\]
commutes up to natural isomorphism.

%%%%%%%%%%%%%%%%%%%%%%%%%%%%%%%%%%%%%%%%%%%%%%%%%%%

\begin{lem}\label{fibrcomoprhmorphthm}
	Let $(\cat,J)$, $(\dcat,K)$, $(\topos,L)$ and $(\ftopos,M)$ be sites and let $A \colon \cat \to \topos$, $B \colon \dcat \to \ftopos$, $F \colon \cat \to \dcat$ and $G \colon \topos \to \mathcal{F}$ be functors.  Suppose that
	\begin{enumerate}
		\item the functors  $A \colon \cat \to \topos$ and $B \colon \dcat \to \ftopos$ are Street fibrations that yield comorphisms of sites
		\[
		A \colon (\cat,J) \to (\topos,L), \ \ B \colon (\dcat,K) \to (\ftopos,M),
		\]
		\item and that the pair $(F,G)$ constitute a morphism of Street fibrations
		\[
		\begin{tikzcd}
			\cat \ar{r}{F} \ar{d}{A} \ar[phantom]{rd}[description]{\cong} & \dcat \ar{d}{B} \\
			\topos \ar{r}{G} & \ftopos,
		\end{tikzcd}
		\]
		and moreover yield morphisms of sites
		\[
		F \colon (\cat,J) \to (\dcat,K), \ \ 
		G \colon (\topos,L) \to (\ftopos,M).
		\]
		Then the induced square of geometric morphisms
		\[\begin{tikzcd}
			\Sh(\cat,J) \ar{d}{C_A} \ar[phantom]{rd}[description]{\cong} & \ar{l}[']{\Sh(F)} \Sh(\dcat,K) \ar{d}{C_B} \\
			\Sh(\topos,L) & \ar{l}[']{\Sh(G)} \Sh(\ftopos,M)
		\end{tikzcd}\]
		commutes up to isomorphism.
	\end{enumerate}
\end{lem}
\begin{proof}
	The overarching method of the proof is to turn the morphisms of sites $F$ and $G$ into comorphisms of sites, and then appeal to the bifunctoriality of sending a comorphism of sites to its induced geometric morphism.  We are able to turn morphisms of sites into comorphisms of sites by \cite[Theorem 3.16]{denseness}.  Namely, for the morphism of sites $F \colon (\cat, J) \to (\dcat,K)$, there are functors
	\[\begin{tikzcd}
		\cat  \ar[shift left]{r}{i_F} & (1_\dcat \downarrow \! F) \ar[shift left]{l}{\pi_\cat} \ar{r}{\pi_\dcat} & \dcat
	\end{tikzcd}\]
	where
	\begin{enumerate}
		\item $(1_\dcat \downarrow \! F)$ denotes the comma category whose objects are pairs
		\[\left(c, d \xrightarrow{a} F(c)\right)\]
		of an object $c \in \cat$ and an arrow $d \to F(c)$ in $\dcat$;
		\item $\pi_\cat \colon (1_\dcat \downarrow \! F) \to \cat$ and $\pi_\dcat \colon (1_\dcat \downarrow \! F) \to \dcat$ are the respective projection functors;
		\item $i_F \colon \cat \to (1_\dcat \downarrow \! F)$ is the functor that sends $c \in \cat$ to 
		\[\left(c,  F(c)  \xrightarrow{\id_{F(c)}} F(c) \right) \in (1_\dcat \downarrow \! F).\]
	\end{enumerate}
	Moreover, when the category $(1_\dcat \downarrow \! F)$ is endowed with the Grothendieck topology $\Tilde{K}$, whose covering sieves are precisely those that are sent by $\pi_\dcat$ to $K$-covering sieves, we have that
	\begin{enumerate}
		\item $\pi_\cat \colon ((1_\dcat \downarrow \! F),\Tilde{K}) \to (\cat,J)$ is a comorphism of sites,
		\item $i_F \colon (\cat,J) \to ((1_\dcat \downarrow \! F),\Tilde{K})$ is a morphism of sites,
		\item $\pi_\dcat \colon ((1_\dcat \downarrow \! F),\Tilde{K}) \to (\dcat,K)$ is both a morphism and comorphism of sites and induces an equivalence of toposes 
		\[\Sh((1_\dcat \downarrow \! F),\Tilde{K}) \simeq \Sh(\dcat,K),\]
	\end{enumerate}
	and also that $\Sh(F) = C_{\pi_\cat} \circ \Sh(\pi_\dcat)$, and $C_{\pi_\dcat}$ is an inverse to $ \Sh(\pi_\dcat)$.  Similarly, there are functors
	\[\begin{tikzcd}
		\topos \ar[shift left]{r}{i_G} & (1_\ftopos \downarrow \! G) \ar[shift left]{l}{\pi_\topos} \ar{r}{\pi_\ftopos} & \ftopos
	\end{tikzcd}\]
	with analogous properties, in particular $\Sh(G) = C_{\pi_\topos} \circ \Sh(\pi_\ftopos)$ and $C_{\pi_\ftopos}$ is an inverse for $\Sh(\pi_\ftopos)$.
	
	We construct a comorphism of sites $H \colon ((1_\dcat \downarrow\! F),\Tilde{K}) \to ((1_\ftopos \downarrow\! G),\Tilde{M})$ such that the diagram
	\begin{equation}\label{commutingcomorphs}
		\begin{tikzcd}
			\cat \ar{d}{A} \ar[phantom]{rd}[description]{\cong} & \ar{l}[']{\pi_\cat} \ar{d}{H} (1_\dcat \downarrow\! F) \ar[phantom]{rd}[description]{\cong} \ar{r}{\pi_\dcat} & \dcat \ar{d}{B} \\
			\topos & \ar{l}[']{\pi_\topos} \ar{r}{\pi_\ftopos} (1_\ftopos\downarrow\! G) & \ftopos
		\end{tikzcd}
	\end{equation}
	commutes up to isomorphism.  Define the functor $H$ as sending an object $\left(c, d \xrightarrow{a} F(c)\right)$ to
	\[\left(A(c), B(d) \xrightarrow{B(a)} B(F(c)) \cong G(A(c)) \right) ,\]
	where we have used that the square
	\begin{equation}\label{fibrcomoprhmorphthm:eq:morph_of_fibr}
		\begin{tikzcd}
			\cat \ar{r}{F} \ar{d}{A} \ar[phantom]{rd}[description]{\cong} & \dcat \ar{d}{B} \\
			\topos \ar{r}{G} & \ftopos,
		\end{tikzcd}
	\end{equation}
	commutes up to isomorphism.  Similarly, $H$ is defined to send an arrow 
	\[ \left(c', d' \xrightarrow{a'} F(c')\right) \xrightarrow{(g,h)} \left(c, d \xrightarrow{a} F(c)\right)\]
	to
	\[  \left(\!A(c'), B(d') \xrightarrow{B(a')} B(F(c')) \cong G(A(c'))\!\right) \xrightarrow{(A(g),B(h))} \left(\!A(c), d \xrightarrow{B(a)} B(F(c)) \cong G(A(c)) \!\right).\]
	The functor $H$ clearly makes the diagram \cref{commutingcomorphs} commute up to isomorphism.
	
	It remains to show that $H$ has the cover lifting property.  Let 
	\[S = \lrset{\left(e_i, f_i \xrightarrow{b} G(e_i)\right) \xrightarrow{(g_i,h_i)} \left(A(c), B(d) \xrightarrow{B(a)} B(F(c)) \cong G(A(c))\right) }{i \in I}\]
	be a $\Tilde{M}$-covering sieve, i.e. $\pi_\ftopos(S) = \lrset{f_i \xrightarrow{h_i} B(d)}{i \in I}$ is $M$-covering.  As $A$ is a fibration, there exists, for each $i \in I$, a (weak) cartesian lifting of $ e_i \xrightarrow{g_i} A(c) \in \topos$ to an arrow $c'\xrightarrow{g'} c \in \cat$.  Since the square \cref{fibrcomoprhmorphthm:eq:morph_of_fibr} is also a morphism of fibrations, the arrow $ F(c') \xrightarrow{F(g')} F(c) \in \dcat$ is cartesian too.  Now we apply the fact that $B $ has the cover lifting property to deduce the existence of a $K$-covering sieve $R$ on $d$ such that $B(R) \subseteq \pi_\ftopos(S)$, i.e. for each $ d' \xrightarrow{k} d$ in $R$, there exists an $i \in I$ such that $B(k)$ factors as
	\[\begin{tikzcd}
		B(d') \ar[bend left]{rrrr}{B(k)} \ar[dashed]{r} & f_i \ar{d}{b} \ar{rrr}{h_i} &[-35pt] && B(d) \ar{d}{B(a)} &[-35pt] \\
		& G(e_i) &[-35pt] \cong B(F(c')) \ar{rr}{B(F(g'))} && B(F(c)) &[-35pt] \cong G(A(c)).
	\end{tikzcd}\]
	As $F(g')$ is cartesian, there is a unique arrow $  d' \xrightarrow{\gamma} F(c') \in \dcat$ making the square
	\[\begin{tikzcd}
		d' \ar{r}{k} \ar{d}{\gamma} & d \ar{d}{a} \\
		F(c') \ar{r}{F(g')} & F(c)
	\end{tikzcd}\]
	commute.  Hence, as $R$ is a $K$-covering sieve,
	\[\lrset{\left(c', d' \xrightarrow{\gamma} F(c')\right) \xrightarrow{(g',k)} \left(c, d \xrightarrow{a} F(c)\right)}{k \in R}\]
	is a $\Tilde{K}$-covering lifting of $S$, whence $H$ is a comorphism of sites
	\[
	H \colon ((1_\dcat \downarrow \! F),\Tilde{K}) \to (
	(1_\ftopos \downarrow \! G), \Tilde{M})
	\]
	as desired.
	
	By the commutation of \cref{commutingcomorphs} up to isomorphism, we deduce that the induced diagram of geometric morphisms
	\begin{equation*}
		\begin{tikzcd}
			\Sh(\cat ,J) \ar{d}{C_A} \ar[phantom]{rd}[description]{\cong} & \ar{l}[']{C_{\pi_\cat}} \ar{d}{C_H} \Sh((1_\dcat \downarrow\! F),\Tilde{K}) \ar[phantom]{rd}[description]{\cong} \ar{r}{C_{\pi_\dcat}} & \Sh(\dcat,K) \ar{d}{C_B} \\
			\Sh(\topos,L) & \ar{l}[']{C_{\pi_\topos}} \ar{r}{C_{\pi_\ftopos}} \Sh((1_\ftopos\downarrow\! G),\Tilde{M}) & \Sh(\ftopos,M)
		\end{tikzcd}
	\end{equation*}
	commutes up to isomorphism too.  Thereby, we conclude that
	\[\begin{split}
		C_A \circ \Sh(F) &= C_A \circ C_{\pi_\cat} \circ \Sh(\pi_\dcat), \\
		& \simeq  C_{\pi_\topos} \circ C_H \circ \Sh(\pi_\dcat), \\
		& = C_{\pi_\topos} \circ \Sh(\pi_\ftopos) \circ C_{\pi_\ftopos} \circ C_H \circ \Sh(\pi_\dcat), \\
		& \simeq \Sh(G) \circ C_B \circ C_{\pi_\dcat} \circ \Sh(\pi_\dcat) , \\
		& = \Sh(G) \circ C_B
	\end{split}\]
	as required.
\end{proof}

\begin{nota}
	In this paper, we will be exclusively concerned with faithful fibrations.  Let $P \colon \topos\op \to \Poset$ be a functor.  By $\topos \rtimes P$ we denote the Grothendieck construction (see \cite[Definition B1.3.1]{elephant} for the general definition), the category which has:
	\begin{enumerate}
		\item as objects, pairs $(e,x)$ where $e$ is an object of $\topos$ and $x$ is an element of $P(e)$,
		\item and an arrow $f \colon (e,x) \to (d,y)$ for each arrow $f \colon e \to d$ in $\topos$ such that $x \leqslant P(f)(y)$.
	\end{enumerate}
	The evident projection functor $p_P \colon \topos \rtimes P \to \topos$ is faithful and a fibration; in fact -- assuming the axiom of choice -- every faithful fibration is of the form $\topos \rtimes P$ for some $P$ (see \cite[\S B1.3]{elephant}).
\end{nota}

%%%%%%%%%%%%%%%%%%%%%%%%%%%%%%%%%%%%%%%%%%

%%%%%%%%%%%%%%%%%%%%%%%%%%%%%%%%%%%%%%%%%%%%%%%%%%%%%

	%%%%%%%%%%%%%%%%%%%%%%%%%%%%%%%%%%%%%%%%%%%%%%%%%%%%%%%
%%%%%%%%%%%%%%%%%%%%%%%%%%%%%%%%%%%%%%%%%%%%%%%%%%%%%%%%%
%%%%%%%%%%%%%%%%%%%%%%%%%%%%%%%%%%%%%%%%%%%%%%%%%%%%%%%%%
\section{Internal locales}\label{sec:intloc}

An \emph{internal locale} of a topos $\topos$ is an object that, according to the internal language of $\topos$, carries the structure of a locale (equivalently, a complete Heyting algebra).

\begin{exs}
	We give some elementary examples.  More will be presented in \cref{sec:exs_of_intloc}.
	\begin{enumerate}
		\item Unsurprisingly, the internal locales of $\sets$, the topos of sets, are just locales.
		\item For any topos $\topos$, the subobject classifier $\Omega_\topos$ is an internal locale of $\topos$.  In fact, we will see in Corollary \ref{terminallocale} that $\Omega_\topos$ is the terminal internal locale in $\topos$.
	\end{enumerate}
\end{exs}

Internal locales can be understood both \emph{internally} and \emph{externally}.  We devote this section to a review of the external treatment of internal locales: that is, given a Grothendieck topos $\topos$ with a site of definition $(\cat,J)$, a classification for which $J$-sheaves $\mathbb{L} \colon \cat\op \to \sets$ correspond to internal locales of $\topos \simeq \Sh(\cat,J)$.

An externalised treatment of internal locales can be found in \cite[\S VI]{JT} and \cite[\S C1.6]{elephant} for the special case when $\cat$ is \emph{cartesian} (i.e. $\cat$ has all finite limits).  When $\cat$ is non-cartesian, \cite[\S 5]{fibredsites} establishes a classification of internal locales of $\Sh(\cat,J)$, which we recall below.

We proceed as follows.
\begin{itemize}
	\item An overview of the classification of internal locales of $\sets^{\cat\op}$, where $\cat$ is an arbitrary category, as calculated in \cite[\S 5]{fibredsites}, is given in \cref{noncartsiteintloc} to \cref{subsec:rel_BC}.  Our exposition is divided into three parts.
	\begin{itemize}
		\item In the first third, we recall some facts regarding localic geometric morphisms.  These correspond up to isomorphism with internal locales.
		\item In the middle third, we give an explicit description for the topos of internal sheaves on an internal locale.
		\item Finally, we recall the \emph{relative Beck-Chevalley condition}, which yields the classification of internal locales given in \cite{fibredsites}.
	\end{itemize}
	We will also observe that this characterisation subsumes the previous result from \cite[\S VI]{JT} for internal locales over a cartesian base category.
	
	\item Finally, \cref{sec:sheafintloc} demonstrates how a classification of the internal locales of $\sets^{\cat\op}$ yields a classification of the internal locales of $\Sh(\cat,J)$.
\end{itemize}

\begin{nota}
	Given a functor $\Lb \colon \cat\op \to \Frm_{\rmopen}$, an object $c$ and an arrow $g$ of $\cat$, when there is no confusion we will use the shorthand $\Lb_c$ for $\Lb(c)$, $g^{-1}$ for $\Lb(g)$ and $\exists_{g}$ for the left adjoint to $\Lb(g)$.
\end{nota}

%%%%%%%%%%%%%%%%%%%
%%%%%%%%%%%%%%%%%%
%%%%%%%%%%%%%%%%%%%%%
\subsection{Localic geometric morphisms}\label{noncartsiteintloc}

Over the next few sections, we re-exposit the classification of internal locales of $\sets^{\cat\op}$ for an arbitrary category $\cat$ as can be found in \cite[\S 5]{fibredsites}.  The `keystone' property used in the classification of internal locales is the connection between internal locales and localic geometric morphisms.

\begin{df}
	A geometric morphism $f \colon \ftopos \to \topos$ is \emph{localic} if every object $F$ of $\ftopos$ is a subquotient of $f^\ast(E)$ for some $E \in \topos$, i.e. there exists $F' \in \ftopos$ and a diagram
	\[\begin{tikzcd}
		F & F' \ar[two heads]{l} \ar[tail]{r} & f^\ast(E).
	\end{tikzcd}\]
\end{df}

Localic geometric morphisms $f \colon \ftopos \to \topos$ correspond bijectively (up to isomorphism) to internal locales of $\topos$ via the following result.

\begin{thm}[\cite{topos}]\label{localicmorph} % or Lemma 1.2 \cite{fact1}, cf. also Proposition 4.2 \cite{fibredsites}
	For a geometric morphism $f \colon \ftopos \to \topos$, the following are equivalent:
	\begin{enumerate}
		\item $f$ is a localic geometric morphism,
		\item $\ftopos$ is the topos of internal sheaves on an internal locale of $\topos$, and moreover this internal locale can be taken as $f_\ast(\Omega_\ftopos)$.
	\end{enumerate}
\end{thm}

This bijection can be visualised with the `bridge' diagram:
\[\begin{tikzcd}
	&& {\ftopos \simeq {\bf Sh}(\cat,J)} \\
	&& {{\begin{matrix}\topos \\ \textit{localic morphism}\end{matrix}}} \\
	{{\begin{matrix}f_\ast(\Omega_\ftopos) \\ \textit{direct image of}\\ \textit{subobject classifier}\end{matrix}}} &&&& {{\begin{matrix}\Lb \in \topos \\ \textit{internal locale.}\end{matrix}}}
	\arrow[""{name=0, anchor=center, inner sep=0}, "f", from=1-3, to=2-3]
	\arrow[curve={height=-24pt}, shorten >=21pt, dashed, tail reversed, from=3-1, to=0]
	\arrow[curve={height=24pt}, shorten >=20pt, dashed, tail reversed, from=3-5, to=0]
\end{tikzcd}\]

Let $\Lb$ be an internal locale of $\topos \simeq \Sh(\cat,J)$.  It appears as the direct image of the subobject classifier $f_\ast(\Omega_\ftopos) \cong \Lb$ for some localic geometric morphism $f \colon \ftopos \to \topos$.  Considered as a sheaf $f_\ast(\Omega_\topos) \colon \topos\op \to \sets$ on the canonical site $(\topos,J_{\rmcan})$ for $\topos$, there is the chain of isomorphisms
\begin{equation*}
	\begin{split}
		f_\ast(\Omega_\ftopos) & \cong \Hom_\topos(-,f_\ast(\Omega_\ftopos)),\\
		&\cong \Hom_\ftopos(f^\ast(-),\Omega_\ftopos), \\
		& \cong \Sub_\ftopos(f^\ast(-))
	\end{split}
\end{equation*}
(here, the first isomorphism is by the Yoneda lemma).  Hence, by composing with the canonical morphism $\ell_\cat \colon \cat \to \Sh(\cat,J)$ (that is, the Yoneda embedding followed by sheafification), we obtain the isomorphism of $J$-sheaves:
\begin{equation}\label{subobjdescr}
	\Lb \cong \Sub_\ftopos(f^\ast\circ \ell_\cat(-)) \colon \cat\op \to \sets.
\end{equation}
Thus, we can observe some basic facts about the internal locale $\Lb$:
\begin{enumerate}
	\item for each object $c$ of $\cat$, $\Lb(c)$ is a complete Heyting algebra, or frame, by \cite[Proposition III.8.1]{SGL};
	\item for each arrow $f \colon c \to d$ of $\cat$, by \cite[Proposition III.8.2]{SGL}, $\Lb(f) \colon \Lb(d) \to \Lb(c)$ is an open frame homomorphism. 
\end{enumerate}
Although not every such functor $\Lb' \colon \cat\op \to \Frm_{\rmopen}$ will yield an internal locale, it is possible to characterise when they do.

%%%%%%%%%%%%%%%%%%%%%%%
%%%%%%%%%%%%%%%%%%%%%%%
%%%%%%%%%%%%%%%%%%%%%%%

%%%%%%%%%%%%%%%%%%%%%%%
%%%%%%%%%%%%%%%%%%%%%%%
%%%%%%%%%%%%%%%%%%%%%%%
\subsection{The topos of internal sheaves}\label{subsec:topos_of_int_sh}

What if we start with an internal locale?  Are we able to present the associated localic geometric morphism with an explicit site?  This is instantiated by constructing the topos of \emph{internal sheaves} on an internal locale.  We review here the explicit construction of a site for this topos given in \cite{fibredsites}.  The construction is obtained by considering progressively smaller dense subsites of the \emph{canonical relative site} of a geometric morphism.

%%%%%%%%%%%%%%%%%%%%%%%%%%%%%%%%%%%%%%
\begin{enumerate}

\item Recall from \cite{denseness} that for each geometric morphism $f \colon \ftopos \to \topos$, there is a \emph{canonical relative site} $(\CanRelFibr,\Tilde{J}_\can)$ for which there is an equivalence
\[
\Sh(\CanRelFibr,\Tilde{J}_\can) \simeq \ftopos.
\]
The underlying category $\CanRelFibr$ can also be written as the comma category $(1_\ftopos \downarrow f^\ast)$: it is the category whose objects are pairs $\left(E, F \to f^\ast(E)\right)$ of an object $E \in \topos$ and an arrow $F \to f^\ast(E) \in \ftopos$.  The Grothendieck topology $\Tilde{J}_\can$ has, as covering sieves, those sieves 
\[\lrset{\left(E_i, F_i \to f^\ast(E_i)\right) \xrightarrow{(f_i,g_i)}\left(E, F \to f^\ast(E)\right)}{i \in I}\]
for which the set of morphisms $\lrset{F_i \xrightarrow{g_i} F}{i \in I}$ is jointly epimorphic (or covering in the canonical topology on $\ftopos$).

\item The full subcategory 
\[\topos \rtimes \Sub_{\ftopos}(f^\ast(-)) \subseteq \CanRelFibr \]
(denoted by $(1_\ftopos \downarrow^{\rm Sub} f^\ast)$ in \cite{fibredsites}) on objects $(E,F \rightarrowtail f^\ast(E))$, where $F$ is a subobject of $f^\ast(E)$, is a $\Tilde{J}_\can$-dense subcategory if and only if $f \colon \ftopos \to \topos$ is a localic geometric morphism (see \cite[Proposition 4.1]{fibredsites}).  Hence, if $f$ is localic, there is an equivalence
\[\ftopos \simeq \Sh(\topos \rtimes \Sub_\ftopos(f^\ast (-)),{\Tilde{J}_\can}|_{\topos \rtimes \Sub_\ftopos(f^\ast (-))}).\]

\item Suppose that $\topos $ is the presheaf topos $ \sets^{\cat\op}$, and let $\Lb \colon \cat\op \to \Frm_{\rmopen}$ be an internal locale of $\topos$.  By $f \colon \ftopos \to \topos$ denote the associated localic geometric morphism for which $\Lb \cong \Sub_{\ftopos}(f^\ast \circ \yo_\cat (-))$.  Since the representable presheaves generate $\sets^{\cat\op}$, we immediately have that the subcategory
\[\cat \rtimes \Lb \simeq \cat \rtimes \Sub_{\ftopos}(f^\ast \circ \yo_\cat (-)) \subseteq  \topos \rtimes \Sub_{\ftopos}(f^\ast(-)) \subseteq \CanRelFibr,\]
is the inclusion of a $\Tilde{J}_\can$-dense subcategory $\cat \rtimes \Lb \subseteq \CanRelFibr$, where we have associated $(c,V) \in \cat \rtimes \Lb$ with the object $(\yo(c),V \rightarrowtail f^\ast(\yo(c))) \in \CanRelFibr$.  Thus, by the comparison lemma (see \cite[\S III.4]{SGA4-3}), there is an equivalence
\[\ftopos \simeq \Sh(\cat \rtimes \Lb, \Tilde{J}_\can|_{\cat \rtimes \Lb}).\]  
\end{enumerate}

%%%%%%%%%%%%%%%%%%%%%%%%

\begin{df}[\cite{fibredsites}]
	Let $\Lb$ be an internal locale of $\sets^{\cat\op}$.  The topos 
	\[\Sh(\cat \rtimes \Lb,\Tilde{J}_\can|_{\rm \cat \rtimes \Lb})\]
	constructed above is called the \emph{topos of internal sheaves} (or just \emph{topos of sheaves}) on $\Lb$.  We will use $K_\Lb$ to denote the Grothendieck topology $\Tilde{J}_\can|_{\cat \rtimes \Lb}$, and will also sometimes denote the topos $\Sh(\cat \rtimes \Lb,K_\Lb)$ by just $\Sh(\Lb)$.  A sieve $S$ in $\cat \rtimes \Lb$ is $K_\Lb$-covering if and only if $S$ contains a small family $\lrset{(c_i,U_i) \xrightarrow{f_i} (d,V)}{ i \in I}$ in $\cat \rtimes \Lb$ such that
	\[V = \bigvee_{i \in I} \exists_{f_i} U_i.\]
\end{df}

%%%%%%%%%%%%%%%%%%%%%%%%%%%%%%%%%%%%

The localic geometric morphism associated to $\Lb$ is also recovered as a site-theoretic construction.  The projection $p_\Lb \colon \cat \rtimes \Lb \to \cat $ yields a comorphism of sites
\[
 p_\Lb \colon (\cat \rtimes \Lb, K_\Lb) \to (\cat, J_\rmtriv).
\]
and since $p_\Lb$ is faithful, the geometric morphism it induces ${C_{p_\Lb} \colon \Sh(\cat \rtimes \Lb, K_\Lb) \to \sets^{\cat\op}}$ is localic (see \cite[Proposition 7.11]{denseness}).

%%%%%%%%%%%%%%

\begin{rem}\label{subobjclassforintloc}
	Let $\Lb$ be an internal locale of $\sets^{\cat\op}$.  The projection $p_\Lb \colon \cat \rtimes \Lb \to \cat$ has a right adjoint $t_\Lb \colon \cat \to \cat \rtimes \Lb$ that sends each object $c \in \cat$ to $(c,\top_c)$.  Therefore, by the description of the direct image functor ${C_{p_\Lb}}_\ast$ found in \cite[Theorem VII.10.4]{SGL}, for each $c \in \cat$, there is an isomorphism of frames
	\[\{\,V \in \Lb_c \mid V\leqslant \top_c \,\} \cong \Lb_c \cong {C_{p_\Lb}}_\ast \left(\Omega_{\Sh(\Lb)}\right)(c) \cong \Omega_{\Sh(L)} \circ t_\Lb (c) \cong \Omega_{\Sh(\Lb)}(c,\top_c).\]
	It is not hard to recognise that this isomorphism can be extended so that, for each object $(c,U)$ of $\cat \rtimes \Lb$, there is an isomorphism
	\[\{\,V \in \Lb_c \mid V\leqslant U \,\} \cong \Omega_{\Sh(\Lb)}(c,U),\]
	and that, for each morphism $(c,U) \xrightarrow{f} (d,W)$ of $\cat \rtimes \Lb$, the transition map
	\[\Omega_{\Sh(\Lb)}(f) \colon \Omega_{\Sh(\Lb)}(d,W) \to \Omega_{\Sh(\Lb)}(c,U)\]
	sends $V \in \Omega_{\Sh(\Lb)}(d,W)$ to $f^{-1}(V) \land U \in \Omega_{\Sh(\Lb)}(c,U)$.
\end{rem}

%%%%%%%%%%%%%%%

%%%%%%%%%%%%%%%
\subsection{The relative Beck-Chevalley condition}\label{subsec:rel_BC}

Given any functor 
\[\Lb \colon \cat\op \to \Frm_{\rmopen},\]
we are still able to define $K_\Lb$ as the function that assigns to each object $(d,V)$ of $\cat \rtimes \Lb$ the collection $K_\Lb(c)$ of sieves $\lrset{(c_i,U_i) \xrightarrow{f_i} (d,V)}{i \in I}$ in $\cat \rtimes \Lb$ such that $V = \bigvee_{i \in I} \exists_{f_i} U_i$.  

However, $K_\Lb$ is not necessarily a Grothendieck topology on $\cat \rtimes \Lb$.  The assignment of sieves $K_\Lb$ clearly satisfies the maximality and transitivity conditions, but $K_\Lb$ does not always satisfy the stability condition (see \cite[Definition III.2.1]{SGL}).

When $K_\Lb$ does define a Grothendieck topology, the topos $\Sh(\cat \rtimes \Lb, K_\Lb)$ is also definable and moreover the geometric morphism 
\[C_{p_\Lb} \colon \Sh(\cat \rtimes \Lb, K_\Lb) \to \sets^{\cat\op},\] induced by the projection $p_\Lb \colon \cat \rtimes \Lb \to \cat$ considered as a comorphism of sites
\[p_\Lb \colon (\cat \rtimes \Lb, K_\Lb) \to (\cat,J_{\rm triv}),\]
is localic by \cite[Proposition 7.11]{denseness}.  Since each fibre has a top element, the functor $p_\Lb$ has a left adjoint $t_\Lb \colon \cat \to \cat \rtimes \Lb$ that sends $c \in \cat$ to the object $(c,\top_c) \in \cat \rtimes \Lb$.  Therefore, the direct image functor ${C_{p_\Lb}}_\ast$ of the induced geometric morphism acts as $-\circ t_\Lb$ by \cite[Theorem VII.10.4]{SGL}.  It is not difficult to calculate, as is done in \cite[\S 5]{fibredsites}, that 
\[\Lb \cong {C_{p_\Lb}}_\ast(\Omega_{\Sh(\cat \rtimes \Lb, K_\Lb)}) \cong \Omega_{\Sh(\cat \rtimes \Lb, K_\Lb)} \circ t_\Lb.\]

In the language of \cite[Definition 5.1]{fibredsites}, if $K_\Lb$ is a Grothendieck topology on $\cat \rtimes \Lb$, then the site $(\cat \rtimes \Lb,K_\Lb)$ is an example of an \emph{existential site}, $K_\Lb$ is an \emph{existential topology} and $\Sh(\cat \rtimes \Lb,K_\Lb)$ is an \emph{existential topos}.  Thus, we arrive at the classification of internal locales in the topos $\sets^{\cat\op}$ established in \cite[\S 5]{fibredsites}.

\begin{df}[\cite{fibredsites}]
	A functor $\Lb \colon \cat\op \to \Frm_{\rmopen}$ is said to satisfy the \emph{relative Beck-Chevalley condition} if, given an arrow $e \xrightarrow{h} d$ of $\cat$, and a sieve $S$ of $\cat \rtimes \Lb$ on the object $(d,V)$ for which $V = \bigvee_{f \in S} \exists_f U$, then
	\[h^{-1}(V) = \bigvee_{g \in h^\ast(S)} \exists_g W,\]
	where $h^\ast(S)$ is the sieve on $(e,h^{-1}(V))$ given by those arrows $(c,W) \xrightarrow{g} (e,h^{-1}(V))$ such that the composite $(c,W) \xrightarrow{g} (e,h^{-1}(V)) \xrightarrow{h} (d,V)$ is in $S$.
\end{df}

\begin{thm}[\cite{fibredsites}]
	Let $\Lb \colon \cat\op \to \Frm_{\rmopen}$ be a functor.  The following are equivalent:
	\begin{enumerate}
		\item $\Lb$ is an internal locale of $\sets^{\cat\op}$,
		\item $\Lb$ satisfies the relative Beck-Chevalley condition,
		\item $K_\Lb$ is a Grothendieck topology on $\cat \rtimes \Lb$.
	\end{enumerate}
\end{thm}

The classification of internal locales of $\sets^{\cat\op}$, when $\cat$ is cartesian can be recovered via the above classification by noting, as is done in \cite[Proposition 5.3]{fibredsites}, that the Beck-Chevalley and relative Beck-Chevalley conditions coincide when $\cat$ has all finite limits (in fact, a study of the proof of \cite[Proposition 5.3]{fibredsites} reveals that only pullbacks are necessary).

\begin{coro}[\cite{fibredsites}]
	Let $\cat$ be a category with all pullbacks.  A functor $\Lb \colon \cat\op \to \Frm_{\rmopen}$ satisfies the relative Beck-Chevalley condition, and thus defines an internal locale of $\sets^{\cat\op}$, if and only if $\Lb$ satisfies the Beck-Chevalley condition: for each pullback square 
	\[\begin{tikzcd}
		c \times_e d \ar{d}{k} \ar{r}{g} & d \ar{d}{h} \\
		c \ar{r}{f} & e
	\end{tikzcd}\]
	of $\cat$, the square
	\[\begin{tikzcd}
		\mathbb{L}_{c \times_e d} \ar{r}{\exists_g} & \mathbb{L}_d \\
		\mathbb{L}_c \ar{r}{\exists_f} \ar{u}{k^{-1}} & \mathbb{L}_e \ar{u}{h^{-1}}
	\end{tikzcd}\]
	commutes.
\end{coro}

Thus, we are able to recover the characterisation for internal locales over a cartesian base category originally given in \cite[Proposition VI.2.2]{JT}.  We complete this discussion with some observations of the Grothendieck topology $K_\Lb$.

\begin{prop}[\cite{fibredsites}]\label{speciesofcov}
	Let $\Lb$ be an internal locale of $\sets^{\cat\op}$.  The Grothendieck topology $K_\Lb$ on $\cat \rtimes \Lb$ is generated by the following two species of covering families:
	\begin{enumerate}[label = (\Alph*)]
		\item\label{covspec1} $\left\{\,(c,U) \xrightarrow{f} (d,\exists_f U)\,\right\}$ for each arrow $c \xrightarrow{f} d$ of $\cat$ and $U \in \Lb_c$,
		\item\label{covspec2} $\lrset{(c,U_i) \xrightarrow{\id_c} \left(c,\bigvee_{i \in I} U_i \right)}{i \in I}$ for object $c$ of $\cat$ and family of opens $U_i \in \Lb_c$, for $i \in I$.
	\end{enumerate}
\end{prop}
\begin{proof}
	We immediately have that both species are $K_\Lb$-covering.  For the converse, note that, given a $K_\Lb$-covering sieve $S$ on $(d,V)$, each morphism $(c,U) \xrightarrow{f} (d,V)$ of $S$ can be written as the composite
	\[\begin{tikzcd}
		(c,U) \ar{r}{f} & (d,\exists_f U) \ar{r}{\id_d} & \left(d,\bigvee_{f \in S} \exists_f U \right) = (d,V).
	\end{tikzcd}\]
	Hence, any Grothendieck topology $J$ for which both species \cref{covspec1} and \cref{covspec2} are $J$-covering contains the Grothendieck topology $K_\Lb$.
\end{proof}

\begin{rem}
	
	Let $\Lb$ be an internal locale of $\sets^{\cat\op}$.  We have refrained from naming the Grothendieck topology $K_\Lb$ the `canonical topology' to avoid confusion, despite it being a generalisation of the canonical topology on a locale.  Unlike a locale $L$ of $\sets$, the Grothendieck topology $K_\Lb$ is \emph{not} necessarily a subcanonical topology (defined on p.~126 of \cite[\S III.4]{SGL}).  Recall from \cite[p.~542-3, \S C1.2]{elephant} that a Grothendieck topology $J$ on a category $\dcat$ is subcanonical only if every $J$-covering sieve $S$ on an object $D$ is \emph{effective-epimorphic}, in the sense that $D$ is the colimit of the (potentially large) diagram
	\[\begin{tikzcd}
		S \ar[hook]{r} & \mathcal{D}/D \ar{r}{U} & \mathcal{D},
	\end{tikzcd}\]
	where $U \colon \mathcal{D}/D \to \dcat$ is the forgetful functor.  Observe, however, that the sieve generated by a $K_\Lb$-covering family $\left\{\,(c,U) \xrightarrow{f} (d,\exists_f U)\,\right\}$ of species \ref{covspec1} is not effective-epimorphic for any non-invertible arrow $f$ of $\cat$ since the colimit in $\cat \rtimes \Lb$ is given by $(c,U)$.
	
\end{rem}

%%%%%%%%%%%%%%%%%%%%%%%%%%%

%%%%%%%%%%%%%%%%%%%%%%%%%%%

%%%%%%%%%%%%%%%%%%%%
%%%%%%%%%%%%%%%%%%%
%%%%%%%%%%%%%%%%%%%%

%%%%%%%%%%%%%%%%%%%%%%%%%%%%%%%
%%%%%%%%%%%%%%%%%%%%%%%%%%%%%%
%%%%%%%%%%%%%%%%%%%%%%%%%%%%%%%
\subsection{Internal locales of sheaf toposes}\label{sec:sheafintloc}

Let $(\cat,J)$ be a Grothendieck site.  The embedding $\Sh(\cat,J) \rightarrowtail \sets^{\cat\op}$ is a localic geometric morphism (see \cite[Example A4.6.2(a)]{elephant}), and thus, for any localic geometric morphism $\ftopos \to \Sh(\cat,J)$, the composite $\ftopos \to \Sh(\cat,J) \rightarrowtail \sets^{\cat\op}$ is still localic since localic geometric morphisms are closed under composition (see \cite[Lemma 1.1]{fact1}).  Therefore, our understanding of the internal locales of the presheaf topos $\sets^{\cat\op}$ can be leveraged to describe the internal locales of $\Sh(\cat,J)$. Similar characterisations are also found in \cite[Proposition 5.10]{fibredsites} and \cite[Corollary C1.6.10]{elephant}.

First, recall from \cite[\S III.3]{SGA4-3} that, given a functor $A \colon \dcat \to \cat$, there is a smallest topology $J_A$ on $\dcat$ making $A$ a comorphism of sites.  In \cite{SGA4-3}, the name \emph{`topologie induite'} was used.  The topology was subsequently dubbed the \emph{Giraud topology} in \cite{relsites} due to its pioneering use in \cite{giraud}.

\begin{lem}\label{intlocaleshf}
	Let $\Lb \colon \cat\op \to \Frm_{\rmopen}$ be a functor indexed over a category $\cat$ with a Grothendieck topology $J$.  The following are equivalent:
	\begin{enumerate}
		\item\label{shflem:1} $\Lb$ is an internal locale of $\Sh(\cat,J)$,
		\item\label{shflem:2} $\Lb$ is an internal locale of $\sets^{\cat\op}$ and a $J$-sheaf,
		\item\label{shflem:3} $K_\Lb$ is stable and contains the Giraud topology $J_{p_\Lb}$,
		\item\label{shflem:4} $K_\Lb$ is stable and there exists a factorisation
		\[\begin{tikzcd}
			& \Sh(\cat \rtimes \Lb,K_\Lb) \ar{d}{C_{p_\Lb}} \ar[dashed]{ld} \\
			\Sh(\cat,J) \ar[tail]{r} & \sets^{\cat\op}.
		\end{tikzcd}\]
	\end{enumerate} 
\end{lem}
The equivalence of statements \ref{shflem:1} and \ref{shflem:2} is a consequence of the fact that the direct image functor of any geometric morphism (in this case the inclusion $\Sh(\cat,J) \hookrightarrow \sets^{\cat\op}$) preserves internal locales (see p. 528 \cite[\S C1.6]{elephant}, c.f. \cite[Corollary C1.6.10]{elephant} as well).  The equivalence of \ref{shflem:2} and \ref{shflem:3} is proved in \cite[Proposition 5.10]{fibredsites} (cf. Remark 5.3(b) \cite{fibredsites} too).  The final equivalence of \ref{shflem:3} and \ref{shflem:4} follows by definition of the Giraud topology.

%%%%%%%%%%%%%%%%%%%%%%%%%%%%%%%%%%%%%%%%%%%%%%%%%%%%%

\section{Examples of internal locales}\label{sec:exs_of_intloc}

We now consider some examples of internal locales over non-cartesian base categories.

\subsection{Gluing internal locales}\label{sec:gluing}

%%%%%%%%%%%%%%%%%%%%%%%%%%%%%%%

What can prevent a functor $\Lb \colon \cat\op \to \Frm_{\rmopen}$ from being an internal locale of $\sets^{\cat\op}$?  What goes wrong when $K_\Lb$ is not stable?  We give an example of such a functor, over a category $\cat$ without all pullbacks, which is not an internal locale, despite $\Lb$ satisfying the Beck-Chevalley condition for those pullbacks in $\cat$ that do exist.  Inspired by this counterexample, we develop in Corollary \ref{gluing} a method for identifying the internal locales of the presheaf topos $\sets^{\dcat\op}$ when $\dcat$ is obtained by `gluing' certain constituent subcategories together.

\begin{ex}\label{nonexample}

	Let $L$ be any locale in $\sets$.  For any category $\cat$ with pullbacks, the constant functor $\Lb \colon \cat\op \to \Frm_{\rmopen}$ for $L$, i.e. $\Lb(c) = L$ and $\Lb(f) = \id_L$ for all objects $c$ and arrows $f$ of $\cat$, satisfies the Beck-Chevalley condition and so defines an internal locale of $\sets^{\cat\op}$.
	
	Now consider the category
	\[\begin{tikzcd}
		\bullet_1 \arrow[loop,  "\id_1 "', distance=2em, in=305, out=235] \ar{r}{f} & \bullet_2 \arrow[loop,  "\id_2 "', distance=2em, in=305, out=235] &  \ar{l}[']{g} \bullet_3 \arrow[loop,  "\id_3"', distance=2em, in=305, out=235]
	\end{tikzcd}\]
	with all arrows displayed (we will refer to it as $\bullet \rightarrow \bullet \leftarrow \bullet$), which clearly lacks a pullback for the diagram
	\[
	\begin{tikzcd}
		& \bullet_3 \ar{d}{g} \\
		\bullet_1 \ar{r}{f} & \bullet_2 .
	\end{tikzcd}
	\]
	The constant functor 
	\[\Lb \colon (\bullet \rightarrow \bullet \leftarrow \bullet)\op \to \Frm_{\rmopen}\]
	for a non-trivial locale $L$ is not an internal locale of $\sets^{(\bullet \rightarrow \bullet \leftarrow \bullet)\op}$.  We can observe that the relative Beck-Chevalley condition fails.  For instance, the set
	\[S= \lrset{ (\bullet_1,U) \xrightarrow{f} (\bullet_2,\top_{\bullet_2}) }{ U \in L }\]
	is a sieve of $(\bullet \rightarrow \bullet \leftarrow \bullet) \rtimes \Lb$ on $(\bullet_2,\top_{\bullet_2})$ for which $\top_{\bullet_2} = \bigvee_S \exists_f U$ but also $\top_{\bullet_3} \neq \bigvee{g^\ast(S)}$, as $g^\ast(S)$ is empty (here $\top_{\bullet_i}$ denotes the top element in $\Lb_{\bullet_i}$).

\end{ex}
%%%%%%%%%%%%%%%%%%%%%%%%%%%

The subobject classifier $\Omega_{\sets^{({\bullet \rightarrow \bullet \leftarrow \bullet})\op}}$ is, of course, an internal locale of the presheaf topos $\sets^{({\bullet \rightarrow \bullet \leftarrow \bullet})\op}$.  Recall (from \cite[\S I.4]{SGL} say) that the subobject classifier $\Omega_{\sets^{({\bullet \rightarrow \bullet \leftarrow \bullet})\op}}$, considered as a diagram of shape $\bullet \rightarrow \bullet \leftarrow \bullet$ in $\Loc_{\rmopen}$, is given by
\[\begin{tikzcd}
	\2 \ar[hook]{r}{i_1} & \2 + \2 & \ar[hook']{l}[']{i_2} \2,
\end{tikzcd}\]
where $\2$ denotes the 2 element locale (i.e. the terminal locale) and $\2 + \2$ is the coproduct in $\Loc$.  This is because there are two sieves, $\emptyset$ and $\{\,\id_1\,\}$, on $\bullet_1$, etc.  Observe that the arrows $i_1$ and $i_2$ are disjoint open embeddings of locales, by which we mean the following are satisfied, for all $V \in \2$:
\[i_1^{-1}\exists_{i_1} V= V, \ \ i_2^{-1} \exists_{i_2} V = \bot, \ \ i_2^{-1}\exists_{i_2} V = V, \ \ i_1^{-1}\exists_{i_2} V = \bot,\]
where $\bot$ represents the bottom element of $\2$.  We show that this property characterises the internal locales of $\sets^{({\bullet \rightarrow \bullet \leftarrow \bullet})\op}$.  We present this as a consequence of a wider theory regarding `gluing' internal locales together.

\begin{coro}\label{gluing}
	Let $\lrset{\cat_i}{i \in I}$ be a set of categories where, for each $i \in I$, $\cat_i$ has a terminal object $\1_i$.  Let $\dcat$ be the category obtained from the disjoint union $\coprod_{i \in I} \cat_i$ by freely adding a new terminal object $\1$.  For each $i \in I$, we denote by $f_i \colon \1_i \to \1$ the newly added morphism connecting the respective terminal objects.  A functor $\Lb \colon \dcat\op \to \Frm_{\rmopen}$ defines an internal locale of $\sets^{\dcat\op}$ if and only if
	\begin{enumerate}
		\item for all $i \in I$,
		\[ \Lb|_{\cat_i} \colon {\cat_i}\op \hookrightarrow \dcat\op \xrightarrow{\Lb} \Frm_{\rmopen}\]
		is an internal locale of $\sets^{\cat_i\op}$,
		\item and, for each pair $i, j \in I$ with $i \neq j$, the locale morphisms
		\[\begin{tikzcd}
			\Lb_{\1_i} \ar{r}{\Lb(f_i)} & \Lb_\1 & \ar{l}[']{\Lb(f_j)} \Lb_{\1_j}
		\end{tikzcd}\]
		are disjoint open embeddings of locales, by which we mean that, for all $V \in \Lb_{\1_i}$, $V' \in \Lb_{\1_j}$,
		\[f_i^{-1}\exists_{f_i} V= V, \ \ f_j^{-1} \exists_{f_i} V = \bot_i, \ \ f_j^{-1}\exists_{f_j} V' = V', \ \ f_i^{-1}\exists_{f_j} V' = \bot_i,\]
		where $\bot_i$ (respectively $\bot_j$) represents the bottom element of $\Lb_{\1_i}$ (resp. $\Lb_{\1_j}$).
	\end{enumerate} 
\end{coro}
\begin{proof}
	For each object $(d,V)$ of $\dcat \rtimes \Lb$, with $d$ being an object of $\cat_j$ say, a sieve $S$ on $(d,V)$ consists only of morphisms contained in $\cat_j \rtimes \Lb|_{\cat_j} \subseteq \dcat \rtimes \Lb$, and any arrow $e \xrightarrow{h} d$ of $\dcat$ is also contained in the subcategory $\cat_j \subseteq \dcat$.  Therefore, we have that $h^{-1}(V) = \bigvee_{g \in h^\ast(S)} \exists_{g} U$ for each such $V$, $S$ and $h$ if and only if $\Lb|_{\cat_j}$ satisfies the relative Beck-Chevalley condition.  We can thus limit our attention to the second criterion of the corollary and sieves on objects of the form $(\1,V) \in \dcat \rtimes \Lb$.

	Suppose that $\Lb$ satisfies the relative Beck-Chevalley condition.  For each $i \in I$ and $U \in \Lb_{\1_i}$, the principle sieve $S$ generated by the arrow $(\1_i,U) \xrightarrow{f_i} (\1,\exists_{f_i} U)$ is $K_\Lb$-covering.  Therefore
	\[f_i^{-1} \exists_{f_i} U = \bigvee_{g \in f_i^\ast(S)} \exists_{g} W = U,\]
	and so $f_i$ is an open embedding.  For each $j \in I$ with $i \neq j$, we have that 
	\[f_j^{-1} \exists_{f_i} U = \bigvee_{g \in f_j^\ast(S)} \exists_{g} W,\]
	which, as $f_j^\ast(S)$ is empty, is equal to $\bot_i$ as required.
	
	Conversely, suppose that $\Lb|_{\cat_i}$ is an internal locale of $\sets^{{\cat_i}\op}$, for each $i \in I$, and that $\Lb(f_i)$ and $\Lb(f_j)$ are disjoint open embeddings for each pair $i, j \in I$ with $i \neq j$.  It remains to show that, if $S$ is a sieve on $(\1,V)$ for which $V = \bigvee_{g \in S} \exists_{g} U$, then 
	\[h^{-1}(V) = \bigvee_{g \in h^\ast(S)} \exists_{g'} U'\]
	for any arrow $e \xrightarrow{h} \1$ of $\dcat$.  It suffices to consider the case when $h = f_j \colon \1_j \to \1$, for some $j \in I$, and $S$ is generated by arrows of the form $(\1_i,U) \xrightarrow{f_i} (\1,V)$, as any arrow $h'$ can be factored as $e \rightarrow \1_j \xrightarrow{f_j} \1$ and any such sieve $S$ can be rewritten as $\lrset{(c,U) \xrightarrow{g} (\1_i, \exists_{g} U) \xrightarrow{f_i} (\1,V) }{f_i \in T}$ where $T$ generates a $K_\Lb$-covering sieve of the desired form.  But now the thesis follows since $\Lb(f_i)$ and $\Lb(f_j)$ are disjoint open embeddings for each pair $i, j \in I$ with $i \neq j$.
\end{proof}

\begin{ex}
	Using Corollary \ref{gluing}, we are instantly able to recognise that a functor
	\[\Lb \colon ({\bullet \rightarrow \bullet \leftarrow \bullet})\op \to \Frm_{\rmopen}\]
	defines an internal locale of the topos $\sets^{({\bullet \rightarrow \bullet \leftarrow \bullet})\op}$ if and only if the diagram in $\Loc$
	\[\begin{tikzcd}
		\Lb_{\bullet_1} \ar{r}{f} &  \Lb_{\bullet_2} & \ar{l}[']{g} \Lb_{\bullet_3}
	\end{tikzcd}\]
	is a pair of disjoint open embeddings, and thus confirm using Corollary \ref{gluing} that the constant functor $\Lb \colon ({\bullet \rightarrow \bullet \leftarrow \bullet})\op \to \Frm_{\rmopen}$ considered in Example \ref{nonexample} does not define an internal locale of  $\sets^{({\bullet \rightarrow \bullet \leftarrow \bullet})\op}$.
	
	More generally, if $\Gamma$ is a \emph{tree} (see \cite[p. 26]{order}), then the internal locales $\sets^\Gamma$ are equivalently functors
	\[
	\Lb \colon \Gamma\op \to \Loc
	\]
	where, for each $x \in \Gamma$, the locale morphisms $\Lb_y \to \Lb_x$ corresponding to the \emph{covers} of $x$ (in the sense of \cite[\S 1.14]{order}) are disjoint open embeddings.
\end{ex}

%%%%%%%%%%%%%%%%%%%%%%%%%%%%%%%%

%%%%%%%%%%%%%%%%%%%%%%%%%%%

%%%%%%%%%%%%%%%%%%%%%%%%%%%%%%%%%%%%

\subsection{Internal locales for monoid actions}

Although every topos has a site whose underlying category has pullbacks (e.g. the canonical site), there are many toposes which have a natural choice of site that lacks pullbacks.  The classification of internal locales given in \cref{sec:intloc} is most aptly applied when studying these toposes.  An important example of such a topos is the topos $\BG$ of representations of a group $G$ (on discrete sets).  This is the presheaf topos $\sets^{G\op}$, where the group $G$ is viewed as a one-object category.

Therefore, applying Theorem \ref{noncartsiteintloc}, we know that an internal locale of $\sets^{G\op}$ is a functor $\Lb \colon G\op \to \Frm_{\rmopen}$ satisfying the relative Beck-Chevalley condition.  But it is easily calculated that any action by $G$ on a locale $L$ by homeomorphisms, i.e. a group homomorphism $G \to \Aut_\Loc(L)$, yields a functor $\Lb \colon G\op \to \Frm_{\rmopen}$ that satisfies the relative Beck-Chevalley condition (this can be deduced as a corollary of the result for monoids below).  Thus, by purely computational means we have recovered the correspondence between internal locales of $\sets^{G\op}$ and $G$-actions on locales that was also observed in \cite[Example C2.5.8(d)]{elephant}.

Conversely, for a monoid $M$, it is not true that any action by $M$ on a locale $L$, i.e. a monoid homomorphism $M \to \End_\Loc(L)$, yields an internal locale of the topos of $M$-sets $\sets^{M\op}$.  Nor will it suffice to restrict to \emph{open} actions, those homomorphism that factor as $M \to \End_{\Loc_{\rmopen}}(L) \subseteq \End_\Loc(L)$.  Instead, an internal locale of $\sets^{M\op}$ must interact stably with respect to the set of \emph{divisors} $\lrset{k \in M}{nk = m}$, for $n , m \in M$, as described below.

\begin{prop}
	Let $M$ be a monoid.  An {open} action of $M$ on a locale $L$ constitutes an internal locale of $\sets^{M\op}$ if and only if, for each $U \in L$ and each pair $n, m \in M$,
	\[
	n^{-1}(\exists_m U) = \bigvee_{\substack{k \in M \\ nk = m}} \exists_k U.
	\]
\end{prop}
\begin{proof}
	We must show that the functor $\Lb \colon M\op \to \Frm_{\rmopen}$ induced by the open action of $M$ on $L$ satisfies the relative Beck-Chevalley condition if and only if, for each $U \in L$ and each pair $n, m \in M$,
	\[
	n^{-1}(\exists_m U) = \bigvee_{\substack{k \in M \\ nk = m}} \exists_k U.
	\]

	Assuming the relative Beck-Chevalley condition, the $K_\Lb$-covering sieve $R$ generated by the single arrow $(\ast, U) \xrightarrow{m} (\ast, \exists_m U)$ must be stable under the map $(\ast, n^{-1} \exists_m U) \xrightarrow{n} (\ast, U)$.  We readily calculate that 
	\[n^\ast(R) = \lrset{(\ast, V) \xrightarrow{k} (\ast, n^{-1} \exists_m U)}{nk = m \text{ and } V \leqslant k^{-1}n^{-1} \exists_m U}.\]
	Hence, we have that
	\[
	n^{-1}(\exists_m U) = \bigvee_{k \in n^\ast(R)} \exists_k V.
	\]
	By the inequality
	\[ V \leqslant  k^{-1}n^{-1} \exists_m U =  k^{-1}n^{-1} \exists_n \exists_k U \leqslant U , \]
	we deduce that $\exists_k V \leqslant \exists_k U$.  Simultaneously, the equality $\exists_n \exists_k U  = \exists_k U$ implies that $ \exists_k U \leqslant n^{-1}( \exists_m U)$.  Combining the two inequalities, we conclude that
	\[
	n^{-1}(\exists_m U) = \bigvee_{k \in n^\ast(R)} \exists_k V = \bigvee_{\substack{k \in M \\ nk = m}} \exists_k U
	\]
	as required.

	For the converse, let $S$ be a sieve on $M \rtimes \Lb$ on $(\ast,V)$ for which $V = \bigvee_{m \in S} \exists_f U$.  Then
	\begin{align*}
		n^{-1}(V) & = \bigvee_{m \in S} n^{-1} \exists_m U, \\
		& = \bigvee_{m \in S} \bigvee_{\substack{k \in M \\ nk = m}} \exists_k U .
	\end{align*}
	We need only finally note that $n^\ast(S) = \lrset{(\ast,U) \xrightarrow{k} (\ast, V)}{\exists m \in S, nk = m}$.
\end{proof}

%%%%%%%%%%%%%%%%%%%%%%%%%%%%%%%
%%%%%%%%%%%%%%%%%%%%%%%%%%%%%%
%%%%%%%%%%%%%%%%%%%%%%%%%%%%%%%

%%%%%%%%%%%%%%%%%%%%%%%%%%%%%%%%%%%%%%%%%%%%%%%%%%%%%

%%%%%%%%%%%%%%%%%%%%%%%%%%%%%%%%%%%%%%%%%%%%%%%%%%%%%%%%%
\section{Internal locale morphisms}\label{intlocmorphsec}

In this section we begin our study the morphisms of internal locales and their properties.  We aim to provide a parallel to the treatment of locale morphisms and the geometric morphisms between localic toposes that is found in \cite[\S IX]{SGL}.  Therein it is shown that, given two locales $X, \, Y$ (of $\sets$), there is an equivalence
\begin{equation}\label{eq:SGLlocandgeo}
	\Loc(X,Y) \simeq \Geom(\Sh(X),\Sh(Y))
\end{equation}
between the category of locale morphisms $X \to Y$ and the category of geometric morphisms ${\Sh(X) \to \Sh(Y)}$.  The morphisms of internal locales (over a cartesian base category) were first characterised in \cite[\S VI.2]{JT}.

\begin{df}[\cite{JT}]
	An \emph{internal locale morphism} $\fb \colon \Lb_1 \to \Lb_2$, between internal locales $\Lb_1,\Lb_2\colon \cat\op \to \Frm_{\rmopen}$ of the topos $\Sh(\cat,J)$, is a natural transformation $\fb^{-1} \colon \Lb_2 \to \Lb_1$ such that, for each object $c$ of $\cat$, $\fb^{-1}_c \colon \Lb_2(c)\to \Lb_1(c)$ is a frame homomorphism and, for each morphism $c \xrightarrow{g} d$ of $\cat$, the diagram
	\[\begin{tikzcd}
		\Lb_2(d) \ar{d}{\fb_d^{-1}} \ar[shift right]{r}[']{\Lb_2(g)} & \Lb_2(c) \ar[shift right]{l}[']{\exists_{\Lb_2(g)}} \ar{d}{\fb_c^{-1}} \\
		\Lb_1(d)\ar[shift right]{r}[']{\Lb_1(g)} & \Lb_1(c) \ar[shift right]{l}[']{\exists_{\Lb_1(g)}}
	\end{tikzcd}\]
	is a \emph{morphism of adjunctions}: that is, $\Lb_1(g) \circ \fb_d^{-1} = \fb_c^{-1} \circ \Lb_2(g)$ and $\exists_{\Lb_1(g)} \circ \fb_c^{-1} = \fb_d^{-1}\circ \exists_{\Lb_2(g)}$.
\end{df}

Being a natural transformation, $\fb^{-1} \colon \Lb_2 \to \Lb_1$ induces a morphism of fibrations
\[
\begin{tikzcd}
	\cat \rtimes \Lb_2 \ar{rr}{\id_\cat \rtimes \fb^{-1}} \ar{rd}[']{p_{\Lb_2}} && \cat \rtimes \Lb_1  \ar{ld}{p_{\Lb_1}} \\
	& \cat. &
\end{tikzcd}
\]
The functor $\id_\cat \rtimes \fb^{-1}$ acts on objects by $(c,U) \mapsto (c,\fb_c^{-1}(U))$.  For notational convenience, we denote the functor $\id_\cat \rtimes \fb^{-1}$ by $\breve{\fb}$.

Our first task is to extend the equivalence \cref{eq:SGLlocandgeo} between internal locale morphisms and geometric morphisms for set-based locales to the internal setting.  This has been demonstrated concretely in \cite[\S 4]{fibredsites}.  Our method here will differ slightly from that of \cite{fibredsites} as we never leave our site of definition $(\cat,J)$.  Instead, over the next few results we will construct a bijective correspondence between:
\begin{enumerate}
	\item the internal locale morphisms $\fb \colon \Lb_1 \to \Lb_2$,
	\item the morphisms of fibrations 
	\[
	\begin{tikzcd}
		\cat \rtimes \Lb_2  \ar{rd}[']{p_{\Lb_2}} \ar{rr}{\breve{\fb}} && 	\cat \rtimes \Lb_1  \ar{ld}{p_{\Lb_1}}\\
		& \cat &
	\end{tikzcd}
	\]
	which yield morphisms of sites
	\[
	\breve{\fb} \colon (\cat \rtimes \Lb_1, K_{\Lb_1}) \to (\cat \rtimes \Lb_2, K_{\Lb_2}),
	\]
	\item and also the geometric morphisms $f \colon \Sh(\Lb_1) \to \Sh(\Lb_2)$ for which the triangle
	\begin{equation}\label{diag:triangleyeah}\begin{tikzcd}[column sep = tiny]
			\Sh({\Lb_1}) \ar{rr}{f} \ar{rd}[']{{C_{p_{\Lb_1}}}} && \Sh({\Lb_2}) \ar{ld}{{C_{p_{\Lb_2}}}} \\
			& \Sh(\cat,J)&
	\end{tikzcd}\end{equation}
	commutes.
\end{enumerate}

Thereby, we will recover the biequivalence
\begin{equation}\label{equivpeekpreview}
	\Loc\left(\Sh(\cat,J)\right) \simeq \mathfrak{Loc}/\Sh(\cat,J).
\end{equation}
(as seen in \cite[Corollary 3.5]{fibredsites}).  Here $\Loc\left(\Sh(\cat,J)\right)$ denotes the bicategory of internal locales of $\Sh(\cat,J)$, their internal locale morphisms and natural transformations between these.  By $\mathfrak{Loc}/\Sh(\cat,J)$ we denote the bicategory whose objects are localic geometric morphisms $f \colon \topos \to \Sh(\cat,J)$, whose 1-cells are commuting geometric morphisms
\[\begin{tikzcd}[column sep = tiny]
	\topos \ar{rr}{g} \ar{dr}[']{f} && \topos' \ar{ld}{f'} \\
	& \Sh(\cat,J) ,&
\end{tikzcd}\]
(the geometric morphism $g$ is also localic by \cite[Lemma 1.1(ii)]{fact1}) and whose 2-cells are natural transformations between these.

Finally, having related internal locale morphisms and geometric morphisms, we turn to a study of their properties.  In \cref{subsec:surj}, we will extend, to the to internal setting, the result \cite[Proposition IX.5.5(i)]{SGL}, which states that a locale morphism $f \colon L \to K$ is an surjective locale morphism if and only if the induced geometric morphism $\Sh(f) \colon \Sh(L) \to \Sh(K)$ between localic toposes is surjective.

%%%%%%%%%%%%%%%%%%%%%%%%%%%
%%%%%%%%%%%%%%%%%%%%%%%%%%%
%%%%%%%%%%%%%%%%%%%%%%%%%%%

\subsection{Internal locale morphisms and geometric morphisms}\label{subsec:intmorphgeomorph}

We first demonstrate two constructions: that each morphism of internal locales induces a geometric morphism that makes the triangle \cref{diag:triangleyeah} commute, and, vice versa, each geometric morphism as in \cref{diag:triangleyeah} induces a morphism of internal locales.  Using this, we then demonstrate the biequivalence \cref{equivpeekpreview}.

\begin{prop}\label{morphinducesgeo}
	Let $\Lb_1, \Lb_2 \colon \cat\op \to \Frm_{\rmopen}$ be internal locales of $\Sh(\cat,J)$.  For each internal locale morphism $\fb \colon \Lb_1 \to \Lb_2$, the morphism of fibrations
	\[
	\begin{tikzcd}
		\cat \rtimes \Lb_2 \ar{rd}[']{p_{\Lb_2}} \ar{rr}{\breve{\fb}} && 	\cat \rtimes \Lb_1  \ar{ld}{p_{\Lb_1}}\\
		& \cat &
	\end{tikzcd}
	\]
	induces a morphism of sites 
	\[
	(\cat \rtimes \Lb_2 , K_{\Lb_2}) \xrightarrow{\breve{\fb}}	(\cat \rtimes \Lb_1 , K_{\Lb_1}).
	\]
	Hence the induced geometric morphism $\Sh(\breve{\fb})$ makes the triangle
	\begin{equation}\label{triangle}\begin{tikzcd}[column sep = tiny]
			\Sh(\cat \rtimes \Lb_1,K_{\Lb_1}) \ar{rr}{\Sh(\breve{\fb})} \ar{rd}[']{C_{p_{\Lb_1}}} && \Sh(\cat\rtimes \Lb_2, K_{\Lb_2}) \ar{ld}{C_{p_{\Lb_2}}} \\
			& \Sh(\cat,J)&
	\end{tikzcd}\end{equation}
	commute.
\end{prop}
\begin{proof}
	We check that the functor $\breve{\fb} \colon \cat \rtimes \Lb_2 \to \cat \rtimes \Lb_1$ defines a morphism of sites by checking that the four conditions explicated in Definition \ref{df:morphsites} are satisfied.
	\begin{enumerate}
		\item To show that $\breve{\fb}$ is cover preserving, it suffices to show that the two generating species of $K_{\Lb_2}$-covering families identified in Proposition \ref{speciesofcov} are sent by $\breve{\fb}$ to $K_{\Lb_1}$-covering families.  Let 
		\[\left\{\, (c,U) \xrightarrow{g} (c,\exists_{\Lb_2(g)} U) \,\right\}\]
		be a $K_{\Lb_2}$-covering family of species \ref{covspec1}.  The family 
		\[ \breve{\fb}\left(\left\{\,  (c,U) \xrightarrow{g} (c,\exists_{\Lb_2(g)} U) \,\right\}\right) = \left\{\, (c,\fb_c^{-1}(U)) \xrightarrow{g} (c,\fb_d^{-1}(\exists_{\Lb_2(g)} U)) \,\right\}\]
		is $K_{\Lb_1}$-covering as $\fb_d^{-1}(\exists_{\Lb_2(g)} U) =\exists_{\Lb_1(g)} \fb_c^{-1}(U)$.  Let 
		\[\lrset{ (c,U_i) \xrightarrow{\id_c} (c,\bigvee_{i \in I} U_i)}{i \in I}\]
		be a $K_{\Lb_2}$-covering family of species \ref{covspec2}.  The family
		\[\hskip-12pt \breve{\fb}\left(\lrset{ (c,U_i) \xrightarrow{\id_c} \left(c,\bigvee_{i \in I} U_i\right) }{ i \in I}\right) = \lrset{ (c,\fb_c^{-1}(U_i)) \xrightarrow{\id_c} \left(c,\fb_c^{-1}\left(\bigvee_{i \in I} U_i\right)\right) }{i \in I}\]
		is $K_{\Lb_1}$-covering since $\fb_c^{-1}$ is a frame homomorphism.

		\item Each object $(c,U)$ of $\cat \rtimes \Lb_1$ has an arrow 
		\[ (c,U) \xrightarrow{\id_c} (c,\top) = (c,\fb_c^{-1}(\top)),\]
		and so $\breve{\fb}$ relatively preserves the terminal object.
		
		\item Next, we show that binary products are relatively preserved.  Given a pair of arrows 
		\[g_1 \colon (d,V) \to (c_1,\fb_{c_1}^{-1}(U_1)), \ g_2 \colon (d,V) \to (c_2,\fb_{c_2}^{-1}(U_2))\]
		of $\cat \rtimes \Lb_2$, we have that 
		\begin{align*}
			V & \leqslant \Lb_2(g)(\fb_{c_1}^{-1}(U_1)) \land \Lb_2(g)(\fb_{c_2}^{-1}(U_2)), \\
			&= \fb_{d}^{-1}(\Lb_1(g)(U_1) \land\Lb_1(g)(U_2)).
		\end{align*}
		Hence, there are the commutative triangles
		\[\begin{tikzcd}[column sep = tiny]
			& (d,V) \ar{ld}[']{\id_d} \ar{rd}{g_1} &\\
			(d,\fb_{d}^{-1}(\Lb_1(g)(U_1) \land\Lb_1(g)(U_2))) \ar{rr}{g_1} && (c_1,\fb_{c_1}^{-1}(U_1)) , \\
			& (d,V) \ar{ld}[']{\id_d} \ar{rd}{g_2} &\\
			(d,\fb_{d}^{-1}(\Lb_1(g)(U_1) \land\Lb_1(g)(U_2))) \ar{rr}{g_2} && (c_2,\fb_{c_2}^{-1}(U_2)) .
		\end{tikzcd}\]

		\item Finally, we demonstrate that equalizers are relatively preserved.  Let 
		\[
		\begin{tikzcd}
			(c',U') \ar[shift left]{r}{f_1} \ar[shift right]{r}[']{f_2} & (c,U)
		\end{tikzcd}
		\]
		be a pair of parallel morphisms of $\cat \rtimes \Lb_2$.  If $g \colon (d,V) \to (c',\fb_{c'}^{-1}(U'))$ is a morphism of $\cat \rtimes \Lb_1$ such that $\breve{\fb}(f_1) \circ g = \breve{\fb}(f_2)\circ g$, then $g \colon (d,V) \to (c',\fb_{c'}^{-1}(U'))$ factors through the morphism $g \colon (d,\Lb_1(g)(\fb_{c'}^{-1}(U'))) \to (c',\fb_{c'}^{-1}(U'))$, which is of the form
		\[\breve{\fb}(g) \colon (d,\fb_d^{-1}(\Lb_2(g)(U'))) \to (c',\fb_{c'}^{-1}(U')).\]
	\end{enumerate}
	Finally, the triangle (\ref{triangle}) commutes by Lemma \ref{fibrcomoprhmorphthm}.
\end{proof}

Let $\Lb_1, \, \Lb_2$ be internal locales of $\Sh(\cat,J)$ with an internal locale morphism $\fb \colon \mathbb{L}_1 \to \mathbb{L}_2$.  We will write $ \Sh(\fb) \colon \Sh({\Lb_1}) \to \Sh({\Lb_2})$ for the geometric morphism $\Sh(\breve{\fb})$ induced as above.  By \cite[Lemma 1.2(ii)]{fact1}, $\Sh(\fb)$ is also a localic geometric morphism.

\begin{prop}\label{geoinducesmorph}
	Let $\Lb_1, \, \Lb_2 \colon \cat\op \to \Frm_{\rmopen}$ be internal locales of $\Sh(\cat,J)$.  Each geometric morphism
	\[f \colon \Sh({\Lb_1}) \to \Sh({\Lb_2})\]
	for which the triangle
	\begin{equation}\label{triangle2}\begin{tikzcd}[column sep = tiny]
			\Sh({\Lb_1}) \ar{rr}{f} \ar{rd}[']{{C_{p_{\Lb_1}}}} && \Sh({\Lb_2}) \ar{ld}{{C_{p_{\Lb_2}}}} \\
			& \Sh(\cat,J)&
	\end{tikzcd}\end{equation}
	commutes induces an internal locale morphism $\fb \colon \Lb_1 \to \Lb_2$ for which $\Sh(\fb) = f$.
\end{prop}
\begin{proof}
	For each object $E$ of $\Sh({\Lb_2})$, the function that sends a subobject $U \rightarrowtail E$ to $f^\ast(U) \rightarrowtail f^\ast(E)$ is a frame homomorphism
	\[f^\ast_E \colon \Sub_{\Sh(\Lb_2)}(E) \to \Sub_{\Sh(\Lb_1)}(f^\ast(E)).\]
	Moreover, for each arrow $g \colon E \to E'$ of $\Sh(\Lb_2)$, the diagram
	\[\begin{tikzcd}
		\Sub_{\Sh(\Lb_2)}(E) \ar[shift right]{r}[']{g^{-1}} \ar{d}{f^\ast_E}& \Sub_{\Sh(\Lb_2)}(E') \ar[shift right]{l}[']{\exists_g} \ar{d}{f^\ast_{E'}} \\
		\Sub_{\Sh(\Lb_1)}(f^\ast(E)) \ar[shift right]{r}[']{f^\ast(g)^{-1}}&
		\Sub_{\Sh(\Lb_1)}(f^\ast(E')) \ar[shift right]{l}[']{\exists_{f^\ast(g)}}
	\end{tikzcd}\]
	commutes (see \cite[p.~496-8]{SGL}).  
	
	Since the triangle \cref{triangle2} commutes, i.e. $f^\ast \circ C_{p_{\Lb_2}}^\ast = C_{p_{\Lb_1}}^\ast$, we have that, for each object $c$ of $\cat$,
	\[f^\ast \circ C_{p_{\Lb_2}}^\ast(\ell_\cat(c)) = C_{p_{\Lb_1}}^\ast(\ell_\cat(c)),\]
	where $\ell_\cat$ denotes the canonical functor $\ell_\cat \colon \cat \to \Sh(\cat,J)$.  We observed in equation \cref{subobjdescr} that
	\[\Lb_1 \cong \Sub_{\Sh(\Lb_1)}(C_{p_{\Lb_1}}^\ast\circ \ell_\cat(-)) \colon \cat\op \to \Frm_{\rmopen},\]
	and similarly for $\Lb_2$.  Hence, the frame homomorphisms
	\[f^\ast_{ C_{p_{\Lb_2}}^\ast(\ell_\cat(c))} \colon \Sub_{\Sh(\Lb_2)}( C_{p_{\Lb_2}}^\ast(\ell_\cat(c))) \to \Sub_{\Sh({\Lb_1})}( C_{p_{\Lb_1}}^\ast(\ell_\cat(c))),\]
	for each object $c$ of $\cat$, collectively define an internal locale morphism $\fb \colon \Lb_1 \to \Lb_2$ where $\fb_c^{-1}(U) = f^\ast(U)$ for each subobject $U \rightarrowtail C_{p_{\Lb_2}}^\ast(\ell_\cat(c))$.
	
	Finally, that $\Sh(\fb) = f$ follows from the description of the inverse image $\Sh(\fb)^\ast$ of a geometric morphism induced by a morphism of sites, as can be found in \cite[Theorem VII.10.2]{SGL}.  Namely, for each $U \in \Lb_2(c)$, we have that
	\[\Sh(\fb)^\ast(\ell_{\cat \rtimes \Lb_2}(c,U)) = \ell_{\cat \rtimes \Lb_1}(c,\fb_c^{-1}(U)) = \ell_{\cat \rtimes \Lb_1}(c,f^\ast(U)) = f^\ast(U).\]
\end{proof}

Thus we obtain our bijective correspondence between between:
\begin{itemize}
	\item the internal locale morphisms $\fb \colon \Lb_1 \to \Lb_2$,
	\item the morphisms of fibrations 
	\[
	\begin{tikzcd}
		\cat \rtimes \Lb_2  \ar{rd}[']{p_{\Lb_2}} \ar{rr}{\breve{\fb}} && 	\cat \rtimes \Lb_1  \ar{ld}{p_{\Lb_1}}\\
		& \cat &
	\end{tikzcd}
	\]
	which yield morphisms of sites
	\[
	\breve{\fb} \colon (\cat \rtimes \Lb_1, K_{\Lb_1}) \to (\cat \rtimes \Lb_2, K_{\Lb_2}),
	\]
	\item and the geometric morphisms $f \colon \Sh(\Lb_1) \to \Sh(\Lb_2)$ for which the triangle
	\[\begin{tikzcd}
		\Sh(\Lb_1) \ar{rr}{f}\ar{rd}[']{C_{p_{\Lb_1}}} && \Sh(\Lb_2)\ar{ld}{C_{p_{\Lb_2}}} \\
		&\Sh(\cat,J) &
	\end{tikzcd}\]
	commutes.
\end{itemize}
We now use this bijective correspondence to establish the biequivalence \cref{equivpeekpreview}, as also performed in \cite[Corollary 3.5]{fibredsites}.

\begin{thm}\label{catofintloc}
	There is an equivalence of 2-categories:
	\[\Loc(\Sh(\cat,J)) \simeq \mathfrak{Loc}/\Sh(\cat,J).\]
\end{thm}
\begin{proof}
	By $\mathfrak{L} \colon \mathfrak{Loc}/\Sh(\cat,J) \to \Loc\left(\Sh(\cat,J)\right)$ denote the (1-)functor that sends a localic geometric morphism $f \colon \topos \to \Sh(\cat,J)$ to the internal locale $f_\ast(\Omega_\topos)$ and a commuting geometric morphism 
	\[\begin{tikzcd}[column sep = tiny]
		\topos \ar{rr}{g} \ar{dr}[']{f} && \topos' \ar{ld}{f'} \\
		& \Sh(\cat,J) &
	\end{tikzcd}\]
	to the internal locale morphism $\mathfrak{g} \colon f_\ast(\Omega_\topos) \to f'_\ast(\Omega_{\topos'})$ induced by Proposition \ref{geoinducesmorph}.
	
	By $\mathfrak{T} \colon \Loc\left(\Sh(\cat,J)\right) \to \mathfrak{Loc}/\Sh(\cat,J)$ denote the functor that sends an internal locale $\Lb$ to the localic geometric morphism $C_{p_\Lb} \colon \Sh(\cat \rtimes \Lb, K_\Lb) \to \Sh(\cat,J)$ and an internal locale morphism $\fb \colon \Lb \to \Lb'$ to $\Sh(\fb)$.  By the isomorphism $\Lb \cong {C_{p_\Lb}}_\ast\left(\Omega_{\Sh(\Lb)}\right)$ and Proposition \ref{geoinducesmorph}, the functors $\mathfrak{L}$ and $\mathfrak{T}$ are mutually inverse.
	
	This 1-equivalence extends to a biequivalence.  One direction of the equivalence
	\[\Hom_{\Loc(\Sh(\cat,J))}(\Lb,\Lb') \cong \Hom_{\mathfrak{Loc}/\Sh(\cat,J)}(\Sh(\Lb),\Sh(\Lb'))\]
	follows from \cite[Remark C2.3.5]{elephant}: every natural transformation between morphisms of sites yields a natural transformation between the induced geometric morphisms.  The other direction is obtained by noting that a natural transformation of inverse image functors
	\[\begin{tikzcd}
		{{\bf Sh}(\mathbb{L})} && {{\bf Sh}(\mathbb{L}'),}
		\arrow[""{name=0, anchor=center, inner sep=0}, "{{\bf Sh}(\mathfrak{f})^\ast}"', curve={height=20pt}, from=1-3, to=1-1]
		\arrow[""{name=1, anchor=center, inner sep=0}, "{{\bf Sh}(\mathfrak{f}')^\ast}", curve={height=-20pt}, from=1-3, to=1-1]
		\arrow["\alpha"', shorten <=3pt, shorten >=3pt, Rightarrow, from=0, to=1]
	\end{tikzcd}\]
	for two internal locale morphisms $\fb, \fb' \colon \Lb \to \Lb'$, induces a natural transformation
	\[\begin{tikzcd}
		{{\rm Sub}_{{\bf Sh}(\mathbb{L})}(-)} && {{\rm Sub}_{{\bf Sh}(\mathbb{L}')}(-).}
		\arrow[""{name=0, anchor=center, inner sep=0}, "{{\bf Sh}(\mathfrak{f})^\ast}"', curve={height=20pt}, from=1-3, to=1-1]
		\arrow[""{name=1, anchor=center, inner sep=0}, "{{\bf Sh}(\mathfrak{f}')^\ast}", curve={height=-20pt}, from=1-3, to=1-1]
		\arrow["\alpha"', shorten <=3pt, shorten >=3pt, Rightarrow, from=0, to=1]
	\end{tikzcd}\]
\end{proof}

\begin{coro}\label{terminallocale}
	The subobject classifier $\Omega_\topos$ of a topos is the terminal object of $\Loc(\topos)$.
\end{coro}
\begin{proof}
	The identity $\id_\topos \colon \topos \to \topos$ is the terminal object of $\mathfrak{Loc}/\topos$ .
\end{proof}

%%%%%%%%%%%%%%%%%%%%%%%%%%%%%%%%%%%%%%%%
%%%%%%%%%%%%%%%%%%%%%%%%%%%%%%%%%%%%%%%
%%%%%%%%%%%%%%%%%%%%%%%%%%%%%%%%%%%%%%%%%%
%%%%%%%%%%%%%%%%%%%%%%%%%%%%%%%%%%%%%%%%%%%

%%%%%%%%%%%%%%%%%%%%%%%%%%%%%%%%%%%%%%%%%%%%%%%%%%%
%%%%%%%%%%%%%%%%%%%%%%%%%%%%%%%%%%%%%%%%%%%%%%%%%%%%
%%%%%%%%%%%%%%%%%%%%%%%%%%%%%%%%%%%%%%%%%%%%%%%%%%%%%

%%%%%%%%%%%%%%%%%%%%%%%%%%%%%%%%%%%%%%%%%%%%%%%%%%%%

%%%%%%%%%%%%%%%%%%%%%%%%%%%%%%%%%%%%%%%%
%%%%%%%%%%%%%%%%%%%%%%%%%%%%%%%%%%%%%%%
%%%%%%%%%%%%%%%%%%%%%%%%%%%%%%%%%%%%%%%%%%
%%%%%%%%%%%%%%%%%%%%%%%%%%%%%%%%%%%%%%%%%%%

%%%%%%%%%%%%%%%%%%%%%%%%%%%%%%%%%%%%%%%%%%%%%%%%%%%%%%%%%%%

\subsection{Surjective internal locale morphisms}\label{subsec:surj}

We now turn to characterising some properties of the geometric morphisms induced by internal locale morphisms.  Recall that a locale morphism $f \colon L \to K$ is a \emph{surjection} if the corresponding frame homomorphism $f^{-1}\colon  K \to L$ is injective.  Recall also that a geometric morphism $f \colon \ftopos \to \topos$ is a \emph{surjection} if the inverse image functor $f^\ast \colon \topos \to \ftopos$ is faithful.  In \cite[Proposition X.5.5(i)]{SGL}, it is shown that a locale morphism $f \colon L \to K$ is surjective if and only if the corresponding geometric morphism $\Sh(f) \colon \Sh(L) \to \Sh(K)$ is surjective.  We extend this to the internal setting, and show that surjections of internal locales can be characterised `pointwise'.

\begin{df}
	Let $\fb \colon \Lb_1 \to \Lb_2$ be an internal locale morphism of $\Sh(\cat,J)$.  We say that $\fb$ is a \emph{surjective internal locale morphism} if, for each $c \in \cat$, $\fb_c^{-1} \colon \Lb_2(c) \to \Lb_1(c)$ is injective.
\end{df}

\begin{prop}\label{prop:surjintlocmorph}
	Let $\fb \colon \Lb_1 \to \Lb_2$ be an internal locale morphism of $\Sh(\cat,J)$.  The following are equivalent:
	\begin{enumerate}
		\item the geometric morphism $\Sh(\fb)$ is a surjective geometric morphism,
		\item $\fb$ is a surjective internal locale morphism.
	\end{enumerate}
\end{prop}
\begin{proof}
	By \cite[Theorem 6.3]{denseness}, the geometric morphism $\Sh(\fb)$ is surjective if and only if the corresponding morphism of sites 
	\[\breve{\fb} \colon (\cat \rtimes \Lb_2 ,K_{\Lb_2}) \to (\cat \rtimes \Lb_1, K_{\Lb_1})\]
	is cover reflecting.  Suppose that each $\fb_d^{-1}$ is injective.  Let $S$ be sieve of $\cat \rtimes \Lb_2$ on $(d,V)$ such that $\breve{\fb}(S)$ is $K_{\Lb_1}$-covering, i.e. $\fb_d^{-1}(V) = \bigvee_{g \in S} \exists_{\Lb_1(g)} \fb_c^{-1}(U)$.  We have that
	\begin{equation*}
		\begin{split}
			\fb_d^{-1}(V)& = \bigvee_{g \in S} \exists_{\Lb_1(g)} \fb_c^{-1}(U), \\
			& = \bigvee_{g \in S} \fb_d^{-1} \exists_{\Lb_2(g)} U , \\
			& = \fb_d^{-1} \left( \bigvee_{g \in S}  \exists_{\Lb_2(g)} U\right).
		\end{split}
	\end{equation*}
	Thus, $V =  \bigvee_{g \in S}  \exists_{\Lb_2(g)} U $ and so $S$ is $K_{\Lb_2}$-covering.  Conversely, if $\breve{\fb}$ is cover reflecting and $\fb_c^{-1}(U) = \fb_c^{-1}(V)$ for a pair $U,V \in \Lb_2(c)$, then $\breve{\fb}$ reflects the maximal cover.  Hence, we conclude that $U = V$.
\end{proof}

%%%%%%%%%%%%%%%%%%%
%%%%%%%%%%%%%%%%%%%%%%

%%%%%%%%%%%%%%%%%%%%%%%%%%%%%%%%%%%%%
%%%%%%%%%%%%%%%%%%%%%%%%%%%%%%%%%%%%%

\section{Internal embeddings and nuclei}\label{sec:emb}

%%%%%%%%%%%%%%%%%%%%%%%%%%%%%%%%%%%
%%%%%%%%%%%%%%%%%%%%%%%%%%%%%%%%%%
%%%%%%%%%%%%%%%%%%%%%%%%%%%%%%%%%%%

This section is dedicated to the study of internal locale embeddings.  Their study is continued in \cref{sec:subtoposes_and_quotients} and \cref{frameintnuc}.  Recall that a locale morphism $f \colon K \to L$ is said to be an \emph{embedding} if the corresponding frame homomorphism $f^{-1} \colon L \to K$ is surjective.  Just as with surjective internal locale morphisms, we define internal locale embeddings as the `pointwise' generalisation.

\begin{df}
		Let $\fb \colon \Lb_1 \to \Lb_2$ be an internal locale morphism of $\sets^{\cat\op}$.  We say that $\fb$ is an \emph{internal locale embedding} if, for each $c \in \cat$, $\fb_c^{-1} \colon \Lb_2(c) \to \Lb_1(c)$ is surjective.  We will also refer to $\Lb_1$ as an \emph{internal sublocale} of $\Lb_2$ and as $\fb$ as the \emph{inclusion} of this internal sublocale.
\end{df}

Recall also that a geometric morphism $f \colon \ftopos \to \topos$ is said to be a \emph{geometric embedding} (and $\ftopos$ a \emph{subtopos} of $\topos$) if the direct image functor $f_\ast$ is full and faithful.  Recall that, by \cite[Proposition IX.5.4]{SGL}, geometric embeddings generalise embeddings of sublocales in the sense that, given a locale morphism $f \colon K \to L$, the induced geometric morphism $\Sh(f) \colon \Sh(K) \to \Sh(L)$ between the toposes of sheaves is a geometric embedding if and only if $f$ is an embedding of locales.  The aim of this section is to prove an analogous result for embeddings of internal locales: that, given a morphism of internal locales $\fb \colon \Lb' \to \Lb$ of $\Sh(\cat,J)$, the geometric morphism $\Sh(\fb)$ is an embedding if and only if $\fb$ is an internal locale embedding.  One direction is easily achieved, as established below, by applying results from \cite[\S 6]{denseness}.  Demonstrating the converse is postponed to \cref{subsec:geomemb}.

\begin{prop}\label{intlocinclusion}
	If $\fb \colon \mathbb{L}_1 \to \mathbb{L}_2$ is an internal locale embedding of $\Sh(\cat,J)$, then $\Sh(\fb)$ is a geometric embedding.
\end{prop}
\begin{proof}
	Let $(\mathcal{D},J)$ be a site and $\topos$ a Grothendieck topos.  By \cite[Corollary 6.4]{denseness}, a $J$-continuous flat functor $G \colon \mathcal{D} \to \topos$ yields a geometric embedding if and only if:
	\begin{enumerate}
		\item each object $E$ of $\topos$ is covered by objects of the form $G(d)$, for $d \in \mathcal{D}$,
		\item for each pair of objects $d, \, d'$ of $\mathcal{D}$ and arrow $g \colon G(d) \to G(d')$ of $\topos$, there exists a family of morphisms $S = \{\,h_i \colon e_i \to d\mid i \in I\,\}$ such that $G(S)$ is jointly epimorphic in $\topos$ and, for each $i \in I$, $g \circ G(h_i) = G(k_i)$ for some arrow $k_i \colon e_i \to d'$ in $\mathcal{D}$.
	\end{enumerate}
	
	The canonical functor $\ell_{\cat \rtimes \Lb_1} \colon \cat \rtimes \Lb_1 \to \Sh(\Lb)$ is a $K_{\Lb_1}$-continuous flat functor that induces an equivalence of toposes, which is in particular an inclusion, and so satisfies both conditions of \cite[Corollary 6.4]{denseness}.  The $K_{\Lb_2}$-continuous flat functor $\Sh(\fb)^\ast \circ \ell_{\cat \rtimes \Lb_2} \colon \cat \rtimes \Lb_2 \to \Sh(\Lb_1)$ corresponding to the geometric morphism $\Sh(\fb)\colon \Sh(\Lb_1) \to \Sh(\Lb_2)$ factors as
	\[\begin{tikzcd}
		\cat \rtimes \Lb_2 \ar{r}{\breve{\fb}} \ar[bend left]{rr}{\Sh(\fb)^\ast \circ \ell_{\cat \rtimes \Lb_2}} & \cat \rtimes \Lb_1 \ar{r}{\ell_{\cat \rtimes \Lb_1}} & \Sh(\Lb).
	\end{tikzcd}\]
	If $\fb_c^{-1}$ is surjective for each $c \in \cat$, then $\breve{\fb} \colon \cat \rtimes \Lb_2 \to \cat \rtimes \Lb_1$ is surjective on both objects and arrows.  Thus, as $\ell_{\cat \rtimes \Lb_1}$ satisfies the conditions of \cite[Corollary 6.4]{denseness}, so too does $\Sh(\fb)^\ast \circ \ell_{\cat \rtimes \Lb_2}$.  Hence, $\Sh(\fb)$ is a geometric embedding as desired.
\end{proof}

To prove the converse, we develop a study of \emph{internal nuclei}.  These are the internal generalisations of the nuclei on a locale, and will appear reminiscent of \emph{Lawvere-Tierney topologies} (a similarity that will be made concrete in Theorem \ref{intsublocthm}).  Nuclei are a useful tool when studying sublocales since many properties of sublocales are more readily proven using nuclei than directly.  In particular, that the sublocales of a locale $L$ form a co-frame is often proved via nuclei, as discussed in \cref{frameintnuc} below.  We proceed as follows.
\begin{itemize}
	\item In \cref{subsec:intnuc}, the notion of an internal nucleus on an internal locale $\Lb$ is introduced and it is shown that internal nuclei correspond bijectively with internal sublocales of $\Lb$.
	\item We show in \cref{subsec:geomemb} that internal nuclei on $\Lb$, and thus by extension internal sublocales of $\Lb$, correspond bijectively with Lawvere-Tierney topologies on $\Omega_{\Sh(\Lb)}$, and hence subtoposes of $\Sh(\Lb)$.
	\item Finally, in \cref{subsec:surj_incl}, we conclude that the surjection-inclusion factorisation of a localic geometric morphism is calculated `pointwise'.
\end{itemize}

Since every geometric embedding is localic (see \cite[Example A4.6.2(a)]{elephant}), every subtopos of the topos of sheaves on an internal locale $\Sh(\Lb)$ is therefore obtained by an internal sublocale of $ \Lb$.  Thus, to understand the subtoposes of $\Sh(\Lb)$, it suffices to study the internal sublocales of $\Lb$.

%%%%%%%%%%%%%%%%%%%%%%%%%%%%%%%%%%
%%%%%%%%%%%%%%%%%%%%%%%%%%%%%%%%%%

%%%%%%%%%%%%%%%%%%%%%%%%%%%%%%%%%%
%%%%%%%%%%%%%%%%%%%%%%%%%%%%%%%%%%

\subsection{Internal nuclei}\label{subsec:intnuc}

Recall from \cite[\S II.2]{stone} that a nucleus on a locale $L$ is a function $j \colon L \to L$ satisfying, for all $x,y \in L$,
\[x \leqslant j(x), \ \ j(j(x)) \leqslant j(x), \ \ j(x \land y) = j(x) \land j(y).\]
These properties are referred to as $j$ being, respectively, \emph{inflationary}, \emph{idempotent}, and \emph{meet-preserving}.  Any function satisfying these properties must also be monotone.

It is well-known (see \cite[Theorem II.2.3]{stone}) that there is a bijective correspondence between nuclei on $L$ and sublocales of $L$.  In one direction, the nucleus associated to a sublocale $f \colon K \rightarrowtail L$ is given by the function $f_\ast f^{-1} \colon L \to L$ (here $f_\ast$ denotes the right adjoint to $f^{-1}$, see Notation \ref{notation}).  Conversely, given a nucleus $j \colon L \to L$, the image of $j$ as a subset of $L$, which we denote by $L^j$, can be given the structure of a frame.  The meets are computed as they are in $L$ while the join of a subset $\{\,U_i \mid i \in I \,\} \subseteq L^j$ is computed as $j\left(\bigvee_{i \in I} U_i \right)$, where $\bigvee_{i \in I} U_i$ is the join in $L$.  It is then clear that $j \colon L \to L^j$ constitutes a surjective frame homomorphism, and hence the inclusion of a sublocale (see \cite[Lemma II.2.2]{stone} or \cite[Proposition IX.4.3]{SGL}).

\begin{df}
		Let $\Lb \colon \cat\op \to \Frm_{\rm open}$ be an internal locale of $\Sh(\cat,J)$.  An \emph{internal nucleus} is a natural transformation $j \colon \Lb \to \Lb$ (as a functor into $\sets$) such that each component $j_c \colon \Lb_c \to \Lb_c$, for $c \in \cat$, is a nucleus on the locale $\Lb_c$.
\end{df}

When the subobject classifier $\Omega_{\Sh(\cat,J)}$ of $\Sh(\cat,J)$ is considered as an internal locale, the definition of an internal nucleus $j \colon \Omega_{\Sh(\cat,J)} \to \Omega_{\Sh(\cat,J)} $ coincides with that of a \emph{Lawvere-Tierney topology} (see \cite[Definition A4.4.1]{elephant}).  For a localic geometric morphism $f \colon \ftopos \to \topos$, we will observe below that internal nuclei on $f_\ast (\Omega_{\ftopos})$ correspond bijectively with Lawvere-Tierney topologies on $\Omega_{\ftopos}$.

First, we establish a bijective correspondence between internal nuclei and internal sublocales that generalises the bijective correspondence for locales (see \cite[Theorem II.2.3]{stone}).

\begin{lem}\label{nucleuslemma}
	Let $j \colon L \to L$ be a nucleus.  For each subset $\{\,U_i \mid i \in I \,\} \subseteq L$, we have that:
	\[j \left(\bigvee_{i \in I} U_i\right) = j \left( \bigvee_{i \in I} j U_i \right).\]
\end{lem}
\begin{proof}
	The first inequality $j \left( \bigvee_{i \in I} U_i \right)\leqslant j \left( \bigvee_{i \in I} j U_i \right)$ is a consequence of $j$ being inflationary as $ U_i \leqslant j U_i$ for each $i \in I$.  The converse inequality is achieved by applying $j$ to both sides of the canonical inequality $\bigvee_{i \in I} jU_i \leqslant j \left( \bigvee_{i \in I} U_i \right)$.
\end{proof}

\begin{prop}\label{intnucgivessubloc}
	Each internal nucleus $j$ on an internal locale $\Lb$ of $\Sh(\cat,J)$ defines an embedding of internal locales $\Lb^j \hookrightarrow \Lb$.
\end{prop}
\begin{proof}
	By the above discussion, for each object $c$ of $\cat$, the nucleus $j_c \colon \Lb_c \to \Lb_c$ induces a sublocale $\Lb_c^j$ of $\Lb_c$.  As $j$ is a natural transformation, for each arrow $c \xrightarrow{g} d$ of $\cat$, $g^{-1} \colon \Lb_d\to \Lb_c$ restricts to a function $g^{-1} \colon \Lb_d^j \to \Lb_c^j$ which, by the definition of meets and joins in $\Lb_d^j$ and $\Lb_c^j$, can easily be shown to be a frame homomorphism.  We must therefore show that each $g^{-1} \colon \Lb_d^j \to \Lb_c^j$ is also open.  
	
	A left adjoint is given by $j_d  \exists_{\Lb(g)}$ since, for each $U \in \Lb^j_c$ and $V \in \Lb^j_d$,
	\[
	j_d \exists_{\Lb(g)} U \leqslant V = j_d(V)  \iff \exists_{\Lb(g)} U \leqslant V \iff U \leqslant g^{-1}(V),
	\]
	and furthermore the Frobenius condition is satisfied:
	\[j_d \exists_{\Lb(g)} U \land V = j_d \exists_{\Lb(g)} U \land j_d V = j_d((\exists_{\Lb(g)} U )\land V) = j_d \exists_{\Lb(g)}( U \land g^{-1}(V)).\]
	We thus conclude that each internal nucleus $j$ induces a functor $\Lb^j \colon \cat\op \to \Frm_{\rm open} $.
	
	Moreover, we observe that the square
	\[\begin{tikzcd}
		\Lb_c \ar{d}{j_c} \ar{r}{\exists_{g}} & \Lb_d \ar{d}{j_d} \\
		\Lb_c^{j_c} \ar{r}{j_d \exists_{g}} & \Lb_d^{j_d}
	\end{tikzcd}\]
	commutes.  For each $U \in \Lb_c$, $U \leqslant j_c(U)$ and so $j_d \exists_{g} U \leqslant j_d \exists_{g} j_c(U)$.  Conversely, as $U \leqslant g^{-1} \exists_{g} U$, it follows that
	\begin{align*}
		j_d(U) \leqslant j_d  g^{-1} \exists_g (U) & \implies j_d(U) \leqslant g^{-1} j_c \exists_g (U), \\
		& \implies \exists_g   j_d(U) \leqslant j_c \exists_g (U), \\
		& \implies j_c \exists_g  j_d(U) \leqslant j_c  \exists_g(U).
	\end{align*}
	Therefore, we have a natural transformation $j \colon \Lb \to \Lb^j$ where each component is a surjective frame homomorphism such that $j_d \exists_g j_c = j_d \exists_g$ for each arrow $d \xrightarrow{g} c$ of $\cat$.  Hence, $j$ would define an embedding of internal locales if $\Lb^j$ were also an internal locale of $\sets^{\cat\op}$.

	To show that $\Lb^j$ is an internal locale, it remains only to show that functor $\Lb^j$ satisfies the relative Beck-Chevalley condition.  Let $S$ be a sieve on $(d,V) \in \cat \rtimes \Lb^j$ such that
	\[V = j_d\left(\bigvee_{g \in S} j_d \exists_{\Lb(g)} U \right),\]
	which, by Lemma \ref{nucleuslemma}, is equal to $j_d\left(\bigvee_{g \in S} \exists_{\Lb(g)} U\right)$, and let $e \xrightarrow{h} d$ be an arrow of $\cat$.  For notational convenience, let $W $ denote $ \bigvee_{g \in S} \exists_{\Lb(g)} U$.  Since $\Lb$ is an internal locale of $\Sh(\cat,J)$,
	\[h^{-1}(W) = \bigvee_{g \in h^\ast( S)} \exists_{\Lb(g)} U.\]
	Thus, by Lemma \ref{nucleuslemma}, we have the desired equality
	\[h^{-1}(V) = h^{-1}(j_d(W)) = j_e(h^{-1}(W)) = j_e\left(\bigvee_{g \in h^\ast( S)} \exists_{\Lb(g)} U\right) = j_e\left(\bigvee_{g \in h^\ast( S)} j_e\exists_{\Lb(g)} U\right),\]
	and therefore $\Lb^j$ is an internal locale of $\sets^{\cat\op}$.  Since $\Sh(\Lb^j) \to \sets^{\cat\op}$ factors as
	\[\begin{tikzcd}
		\Sh(\Lb^j) \ar{r} & \Sh(\Lb) \ar{r} & \Sh(\cat,J) \ar[tail]{r} & \Sh(\cat,J),
	\end{tikzcd}\] 
	we conclude that $\Lb^j$ is an internal locale of $\Sh(\cat,J)$ too by Lemma \ref{intlocaleshf}.
\end{proof}

\begin{coro}\label{intsublocandnuclei}
	Let $\Lb \colon \cat\op \to \Frm_{\rm open} $ be an internal locale of $\Sh(\cat,J)$.  There is a bijective correspondence between internal sublocales of $\Lb$ and internal nuclei on $\Lb$.
\end{coro}
\begin{proof}
	By the theory of standard locales, there is a bijective correspondence between collections of nuclei
	\[\{\,j_c \colon \Lb_c \to \Lb_c\mid c \in \cat \,\}\]
	and collections of sublocales 
	\[\{\,f_c \colon \Lb_c' \rightarrowtail \Lb_c \mid c \in \cat\,\},\]
	where both are indexed by the objects of $\cat$.  Our bijection will be a restriction of this correspondence.
	
	We have already seen in Proposition \ref{intnucgivessubloc} that if the collection $\{\,j_c \colon \Lb_c \to \Lb_c\mid c \in \cat \,\}$ of nuclei is natural in $c$, i.e. it defines an internal nucleus, then the corresponding collection of sublocales yields an internal sublocale embedding.  It remains to show the other direction: that if $\{\,\fb_c \colon \Lb_c' \rightarrowtail \Lb_c \mid c \in \cat\,\}$ are the components of an internal sublocale embedding, then the corresponding collection of nuclei is natural.

	Let $\fb \colon \Lb' \to \Lb$ be an embedding of an internal sublocale.  Since each component $\fb_c^{-1} \colon  \Lb'(c) \to \Lb(c)$ is surjective, it induces a nucleus ${\fb_\ast}_c \fb^{-1}_c \colon \Lb(c) \to \Lb(c)$, for each object $c$ of $\cat$.  We wish to show that, for each arrow $c \xrightarrow{g} d$ of $\cat$, the square
	\[\begin{tikzcd}
		\Lb_d \ar{d}[']{{\fb_\ast}_d \fb^{-1}_d} \ar{r}{g^{-1}} & \Lb_c \ar{d}{{\fb_\ast}_c \fb^{-1}_c} \\
		\Lb_d \ar{r}{g^{-1}} & \Lb_c
	\end{tikzcd}\]
	commutes.  Since the square 
	\[\begin{tikzcd}
		\Lb_d \ar{d}[']{\fb_d^{-1}} \ar[shift right]{r}[']{g^{-1}} & \Lb_c \ar[shift right]{l}[']{\exists_{g}} \ar{d}{\fb_c^{-1}} \\
		\Lb_d \ar[shift right]{r}[']{g^{-1}} & \Lb_c \ar[shift right]{l}[']{\exists_{g}},
	\end{tikzcd}\]
	is a morphism of adjunctions, taking the respective right adjoints also yields a morphism of adjunctions
	\[\begin{tikzcd}
		\Lb_d \ar[shift left]{r}{g^{-1}} &\ar[shift left]{l}{g_\ast} \Lb_c \\
		\Lb_d \ar{u}{{\fb_\ast}_d} \ar[shift left]{r}{g^{-1}}& \Lb_c. \ar{u}{{\fb_\ast}_c} \ar[shift left]{l}{g_\ast}
	\end{tikzcd}\]
	Hence we have the desired equality
	\[ {\fb_\ast}_c \fb^{-1}_c g^{-1} = {\fb_\ast}_c g^{-1}\fb^{-1}_d = g^{-1} {\fb_\ast}_d \fb^{-1}_d . \]
\end{proof}

%%%%%%%%%%%%%%%%%%%%%%%%%%%%%%%%%%
%%%%%%%%%%%%%%%%%%%%%%%%%%%%%%%%%

\subsection{Geometric embeddings}\label{subsec:geomemb}

We now establish a bijective correspondence between internal nuclei and Lawvere-Tierney topologies, and hence between internal sublocales and subtoposes.  Let $\ftopos$ be a Grothendieck topos.  Recall, from \cite[\S A4.4]{elephant} say, that a \emph{Lawvere-Tierney topology} is a endomorphism $j \colon \Omega_{\ftopos} \to \Omega_{\ftopos}$ on the subobject classifier of the topos $\ftopos$ such that the diagrams
\[\begin{tikzcd}
	\1 \ar{r}{\top} \ar{rd}[']{\top} & \Omega_{\ftopos} \ar{d}{j} & \Omega_{\ftopos} \ar{r}{j} \ar{rd}[']{j} & \Omega_{\ftopos} \ar{d}{j} & \Omega_{\ftopos} \times \Omega_{\ftopos} \ar{d}{j \times j} \ar{r}{\land} & \Omega_{\ftopos} \ar{d}{j} \\
	& \Omega_{\ftopos}, & & \Omega_{\ftopos}, & \Omega_{\ftopos} \times \Omega_{\ftopos} \ar{r}{\land} & \Omega_{\ftopos}
\end{tikzcd}\]
commute.  Recall also that there is a bijection between Lawvere-Tierney topologies and subtoposes of $\ftopos$.  As observed in \cite[Corollary IX.6.6]{SGL}, given a locale $L$, there is a bijection between Lawvere-Tierney topologies on $\Omega_{\Sh(L)}$ (and hence subtoposes of $\Sh(L)$) and nuclei on $L$ (and hence sublocales of $L$).  The following result extends this bijection to the internal setting.

\begin{thm}\label{intsublocthm}
	Let $\Lb \colon \cat\op \to \Frm_{\rm open}$ be an internal locale of $\topos \simeq \Sh(\cat ,J)$.  There is a bijective correspondence between the following:
	\begin{enumerate}
		\item\label{subtopoi} the subtoposes of $\ftopos \simeq \Sh(\Lb)$;
		\item\label{list:nuclei} internal nuclei on $\Lb$;
		\item\label{list:subloc} internal sublocales of $\Lb$.
	\end{enumerate}
	In particular, if $\fb \colon \Lb' \to \Lb$ is an internal locale morphism, $\Sh(\fb)$ is a geometric embedding if and only if $\fb$ is an internal locale embedding.
\end{thm}
\begin{proof}
	The bijective correspondence between internal nuclei and internal sublocales was shown in Corollary \ref{intsublocandnuclei}.  We now demonstrate a bijective correspondence between subtoposes of $\ftopos \simeq \Sh(\Lb)$ and internal nuclei on $\Lb$.  We rely on the characterisation of subtoposes of $\Sh(\Lb)$ in terms of Lawvere-Tierney topologies.
	
	Let $j \colon \Omega_{\Sh(\Lb)} \to \Omega_{\Sh(\Lb)}$ be a Lawvere-Tierney topology and let $f \colon \Sh(\Lb) \to \sets^{\cat\op}$ be the localic geometric morphism such that $f_\ast(\Omega_{\Sh(\Lb)}) \cong \Lb$, i.e. $f \cong C_{p_\Lb}$.  By now applying the direct image functor $f_\ast \colon \Sh(\Lb) \to \sets^{\cat\op}$, we obtain an endomorphism
	\[f_\ast j \colon f_\ast (\Omega_{\Sh(\Lb)}) \cong \Lb \to f_\ast (\Omega_{\Sh(\Lb)}) \cong \Lb.\]
	By the description of ${C_{p_\Lb}}_\ast$ afforded by \cite[Theorem VII.10.2]{SGL}, we have that $(f_\ast j)_c = (j \circ t_\Lb)_c = j_{(c,\top)}$.  We claim that $f_\ast j$ is an internal nucleus.  Since $j$ was a Lawvere-Tierney topology, $f_\ast j$ makes the following diagrams commute:
	\[\begin{tikzcd}\Lb \ar{r}{f_\ast j} \ar{rd}[']{f_\ast j} &  \Lb \ar{d}{f_\ast j} && \Lb \times \Lb \ar{d}[']{f_\ast j \times f_\ast j} \ar{r}{\land} & \Lb \ar{d}{f_\ast j} \\
		& \Lb , && \Lb \times \Lb \ar{r}{\land} & \Lb.
	\end{tikzcd}\]
	Thus, $f_\ast j \colon \Lb \to \Lb$ is a natural transformation such that $(f_\ast j)_c \colon \Lb_c \to \Lb_c$ is idempotent and preserves binary meets, for each $c \in \cat$.  It remains to show that $(f_\ast j)_c$ is inflationary.

	Let $U \in \Lb_c $.  As $j$ is a Lawvere-Tierney topology and natural, there is a commutative diagram of sets
	\[\begin{tikzcd}
		\1(c,U) \ar{r}{\top_{(c,U)}} \ar{rd}[']{\top_{(c,U)}} & \Omega_{\Sh(\Lb)}(c,U) \ar{d}{j_{(c,U)}} & \ar{l}[']{-\land U} \Omega_{\Sh(\Lb)}(c,\top) \ar{d}{(f_\ast j)_c = j_{(c,\top)}}  \\
		& \Omega_{\Sh(\Lb)}(c,U)  & \ar{l}[']{-\land U} \Omega_{\Sh(\Lb)}(c,\top).
			\end{tikzcd}\]
	The displayed morphisms act as follows:
	\begin{enumerate}
		\item the map $\top_{(c,U)} \colon \1(c,U) \to \Omega_{\Sh(\Lb)}(c,U)$ picks out the top element 
		\[U \in \Omega_{\Sh(\Lb)}(c,U) \cong \Sub_{\Sh(\Lb)}(\ell_{\cat \rtimes \Lb}(c,U)),\]
		\item while the map $\Omega_{\Sh(\Lb)}(c,\top) \to \Omega_{\Sh(\Lb)}(c,U)$ is induced by pulling back subobjects along the monomorphism $\ell_{\cat \rtimes \Lb}(c,U) \rightarrowtail \ell_{\cat \rtimes \Lb}(c,\top)$  -- in other words, it sends a subobject $V \in \Omega_{\Sh(\Lb)}(c,\top) \cong \Sub_{\Sh(\Lb)}(\ell_{\cat \rtimes \Lb}(c,\top))$ to $U \land V \in \Omega_{\Sh(L)}(c,U) \cong \Sub_{\Sh(\Lb)}(\ell_{\cat \rtimes \Lb}(c,U))$.
	\end{enumerate}
	Thus, by chasing the element $U \in \Omega_{\Sh(L)}(c,\top)$ through the diagram, we observe that $U \land (f_\ast j)_c(U) = j_U(U) = U$.  Thus, $U \leqslant (f_\ast j)_c(U)$ as desired.

	Conversely, given an internal nucleus $k \colon \Lb \cong f_\ast(\Omega_{\Sh(\Lb)} ) \to \Lb \cong f_\ast(\Omega_{\Sh(\Lb)})$, we define a natural endomorphism $k^f$ on the subobject classifier $\Omega_{\Sh(\Lb)}$, viewed as a sheaf on the site $(\cat \rtimes \Lb, K_\Lb)$, by setting $k^f_{(c,U)} (V)$ as $k_c(V) \land U$, for each $(c,U) \in \cat \rtimes \Lb$ and each $V \in \Omega_{\Sh(\Lb)} (c,U) = \{\, V \in \Lb_c \mid V \leqslant U\,\}$.   We now demonstrate that $k^f$ defines an Lawvere-Tierney topology.
	
	As $k$ is an internal nucleus, by a simple diagram chase it is clear that, for each object $(c,U) \in \cat \rtimes \Lb$, the diagrams
	\[\begin{tikzcd}
		\1(c,U) \ar{r}{\top_{(c,U)}} \ar{rd}[']{\top_{(c,U)}} & \Omega_{\Sh(\Lb)}(c,U) \ar{d}{k^f_{(c,U)}} & \Omega_{\Sh(\Lb)}(c,U) \ar{r}{k^f_{(c,U)}} \ar{rd}[']{k^f_{(c,U)}} & \Omega_{\Sh(\Lb)}(c,U) \ar{d}{k^f_{(c,U)}} \\ 
		& \Omega_{\Sh(\Lb)}(c,U), && \Omega_{\Sh(\Lb)}(c,U),
	\end{tikzcd}\]
	\[\begin{tikzcd}
		\Omega_{\Sh(\Lb)} \times \Omega_{\Sh(\Lb)}(c,U) \ar{r}{\land} \ar{d}[']{k^f_{(c,U)} \times k^f_{(c,U)}} & \Omega_{\Sh(\Lb)}(c,U) \ar{d}{k^f_{(c,U)}} \\
		\Omega_{\Sh(\Lb)} \times \Omega_{\Sh(\Lb)}(c,U) \ar{r}{\land}  & \Omega_{\Sh(\Lb)}(c,U) 
	\end{tikzcd}\]
	all commute.  It remains to observe that $k^f$ is natural.  Since each arrow $(c,U) \xrightarrow{g} (d,V)$ of $\cat \rtimes \Lb$ can be factored as 
	\[(c,U) \xrightarrow{\id_c} (c,g^{-1} (V)) \xrightarrow{g} (d,V) ,\] 
	it suffices to show that both squares in the diagram
	\[\begin{tikzcd}
		\Omega_{\Sh(\Lb)}(d,V) \ar{r} \ar{d}{k^f_{(d,V)}} & \Omega_{\Sh(\Lb)}(c,g^{-1}(V)) \ar{r} \ar{d}{k^f_{(c,g^{-1}(V))}} & \Omega_{\Sh(\Lb)}(c,U) \ar{d}{k^f_{(c,U)}}\\
		\Omega_{\Sh(\Lb)}(d,V) \ar{r} & \Omega_{\Sh(\Lb)}(d,g^{-1}(V)) \ar{r} & \Omega_{\Sh(\Lb)}(c,U) 
	\end{tikzcd}\]
	commute.  
	\begin{enumerate}
		\item The left-hand square commutes since, for each $W \in \Lb_d$,
		\[k_c(g^{-1}(W)) \land g^{-1}(V) = g^{-1}(k_d(W)) \land g^{-1}(V) = g^{-1}(k_d(W) \land V).\]
		\item Meanwhile, the right-hand square commutes since, for each $W \in \Lb_c$,
		\[k_c(W \land U) \land U = k_c(W) \land U = k_c(W\land g^{-1}(V)) \land U .\]
	\end{enumerate}

	Finally, the bijection is completed by noting that the two constructions are mutually inverse.  That is, for each $c \in \cat$ and $U,V \in \Lb_c$,
	\[(f_\ast k^f)_c(V) = k^f_{(c,\top)}(V) = k_c(V) \land \top = k_c(V) \]
	and
	\[(f_\ast j)^f_{(c,U)}(V) = j_{(c,\top)}(V) \land U = j_{(c,U)}(V),\]
	for each internal nucleus $k$ on $\Lb$ and each Lawvere-Tierney topology $j$ on $\Omega_{\Sh(\Lb)}$.
\end{proof}

\subsection{The surjection-inclusion factorisation}\label{subsec:surj_incl}

Recall that every locale morphism $f \colon L \to K$ can be factored uniquely (up to isomorphism) as a surjection of locales followed by an inclusion of locales (see \cite[\S IX.4]{SGL}).  The same is true for geometric morphisms: every geometric morphism $f \colon \ftopos \to \topos$ can be factored as a geometric surjection composed with an inclusion of a subtopos (see \cite[Theorem A4.2.10]{elephant}).  If $f$ is induced by an internal locale morphism, a simple application of Proposition \ref{prop:surjintlocmorph} and Theorem \ref{intsublocthm} yields the following.

\begin{coro}
	Let $\fb \colon \Lb' \to \Lb$ be an internal locale morphism of $\Sh(\cat,J)$.  The surjection-inclusion factorisation of the geometric morphism $\Sh(\fb) \colon \Sh(\Lb') \to \Sh(\Lb)$ is induced by the `pointwise' surjection-inclusion factorisation of $\fb$.
\end{coro}
\begin{proof}
	Let
	\[
	\begin{tikzcd}
		\Sh(\Lb') \ar[two heads]{r} & \Sh(\Lb^{\fb_\ast \fb^{-1}}) \ar[tail]{r} & \Sh(\Lb)
	\end{tikzcd}
	\]
	denote the surjection inclusion factorisation of $\Sh(\fb)$.  By Proposition \ref{prop:surjintlocmorph} and Theorem \ref{intsublocthm}, the factor $\Sh(\Lb') \twoheadrightarrow \Sh(\Lb^{\fb_\ast \fb^{-1}})$ is induced by a surjective internal locale morphism $\Lb' \twoheadrightarrow \Lb^{\fb_\ast \fb^{-1}}$, while $\Sh(\Lb^{\fb_\ast \fb^{-1}}) \rightarrowtail \Sh(\Lb)$ is induced by an internal embedding of locales $\Lb^{\fb_\ast \fb^{-1}} \rightarrowtail \Lb$.  Since internal surjections and embeddings are computed `pointwise', the component at $c \in \cat$ of these internal locale morphisms must agree with the `pointwise' surjection-inclusion factorisation of the locale morphism $\fb_c \colon \Lb'_c \to \Lb_c$.
\end{proof}
%%%%%%%%%%%%%%%%%%%%%%%%%%%%%%%%%%%%%%%%

%%%%%%%%%%%%%%%%%%%%%%%%%%%%%%%%%%%%%%%%%%%%%

\section{An application to categorical logic}\label{sec:subtoposes_and_quotients}

The perspective afforded to topos theory by internal locales can be a powerful tool.  As a demonstration, using the theory of internal sublocales developed above, we give an elegant proof of the well-known fact that the \emph{quotient theories} of a geometric theory are in equivalence with the subtoposes of the \emph{classifying topos} of the theory (see \cite[Examples B4.2.8(i)]{elephant} and \cite[Theorem 3.2.5]{TST}).

\subsection{Geometric theories as internal locales}

We first recall that geometric first-order theories can be understood as internal locales over certain base categories.  For an introduction to geometric first-order logic, the reader is directed to \cite[\S D1]{elephant}.  

\begin{df}
	Let $\Sigma$ be a signature.  We denote by $\Con$ the \emph{category of contexts} for $\Sigma$, the category:
	\begin{enumerate}
		\item whose objects are contexts $\vec{x}$, i.e. finite tuples of free variables,
		\item and whose arrows are \emph{relabellings} $\sigma \colon \vec{y} \to \vec{x}$, i.e. functions of finite tuples that preserve sorts of the variables.
	\end{enumerate} 
\end{df}

If $\Sigma$ is a single sorted signature, then $\Con \simeq \Finsets$.  More generally, if $\Sigma$ has $N$ many sorts, then $\Con$ is the full subcategory of $N \times \Finsets$ on objects $(Z_k)_{k \in N}$ where only finitely many $Z_k$ are non-empty.  Note that $\Con$ has all pushouts (these are computed pointwise).

\begin{df}\label{df:FT}
	Let $\theory$ be a single-sorted, geometric theory over a signature $\Sigma$.  We denote by
	\[ F^\theory \colon \Con \to \Frm \]
	the functor that acts as follows.
	\begin{enumerate}
		\item For each context $\vec{x}$, $F^\theory(\vec{x})$ is the frame:
		\begin{enumerate}
			\item whose elements are $\theory$-provable equivalence classes of formulae in the context $\vec{x}$ -- written as $\form{\varphi}{x}_\theory$,
			\item and the order relation is given by provability in $\theory$, i.e. $\form{\varphi}{x}_\theory \leqslant \form{\psi}{x}_\theory$ if and only if $\theory$ proves the sequent $\varphi \vdash_{\vec{x}} \psi$.
		\end{enumerate}
		\item For each relabelling $\sigma \colon \vec{y} \to \vec{x}$, the frame homomorphism $F^\theory(\sigma)$ acts by sending a formula $\form{\psi}{y}_\theory \in F^\theory(\vec{y})$ to the formula $\form{\psi[\vec{x}/_\sigma \vec{y}]}{x}_\theory$, the formula obtained by simultaneously replacing each instance of the variable $y_i \in \vec{y}$ by $\sigma(y_i) \in \vec{x}$ (since contexts are assumed to be disjoint, we can simultaneously replace free variables).
	\end{enumerate}
\end{df}

This represents the theory $\theory$ as a \emph{doctrine} in the style of Lawvere \cite{adjoint}.  We can consult \cite{seely} to further deduce that, for each relabelling $\sigma \colon \vec{y} \to \vec{x}$, the map $F^\theory(\sigma)$ has a left adjoint $\exists_{F^\theory(\sigma)} \colon F^\theory(\vec{x}) \to F^\theory(\vec{y})$.  This is the function which sends $\form{\varphi}{x}_\theory \in F^\theory(\vec{x})$ to the element
\[
\textstyle \left(\,\vec{y} \, , \, \exists \vec{x} \ \varphi \land \bigwedge_{y_i \in \vec{y}} y_i = \sigma(\vec{y_i}) \, \right)_\theory \in F^\theory(\vec{y}).
\]
Moreover, these left adjoints satisfy the Frobenius and Beck-Chevalley conditions.  Therefore, $F^\theory$ defines an internal locale of $\sets^{\Con}$.  Indeed, by \cite[\S 6]{seely}, every internal locale of $\sets^{\Con}$ is, up to isomorphism, of the form $F^{\theory'}$ for some geometric theory $\theory'$.  This is precisely the observation, in the single-sorted case, of \cite[Theorem D3.2.5]{elephant}, and can also be deduced from the theory of \emph{localic expansions} of \cite[\S 7]{TST}.

\subsection{Quotient theories}

We now turn to proving the equivalence between quotient theories and subtoposes found in \cite[Examples B4.2.8(i)]{elephant} and \cite[Theorem 3.2.5]{TST}.  In the author's estimation, the perspective of internal locales yields the simplest demonstration of this equivalence.

First, recall that the \emph{classifying topos} of the geometric theory $\theory$ is a topos $\topos_\theory$ for which there is a natural equivalence
\[
\Tmod(\ftopos) \simeq \Geom(\ftopos,\topos_\theory)
\]
between the category of models of $\theory$ in a topos $\ftopos$, and their homomorphisms, and the geometric morphisms from $\ftopos$ into $\topos_\theory$ and their natural transformations.  Intuitively, this expresses that, if a topos is a `generalised space', then the classifying topos of a theory is the `generalised space' whose points are models of the theory.  Recall also that the topos $\Sh(F^\theory)$ of internal sheaves on $F^\theory$ classifies $\theory$ in that there is an equivalence $\Sh(F^\theory) \simeq \topos_\theory$ (see \cite{myself-presentation}).

\begin{df}
	Let $\theory$ be a geometric theory over a signature $\Sigma$.
	\begin{enumerate}
		\item A \emph{quotient theory} of $\theory$ is a geometric theory $\theory'$ over the same signature $\Sigma$ as $\theory$ and which contains the axioms of $\theory$.
		
		\item Two quotient theories $\theory', \theory''$ of $\theory$ are said to be \emph{syntactically equivalent}, written as $\theory' \equiv_s \theory''$, if the axioms of $\theory'$ are provable by the theory $\theory''$ and vice-versa.
	\end{enumerate}
\end{df}

\begin{coro}
	Let $\theory$ be a geometric theory.  There is a bijective correspondence between the $\equiv_s$-equivalence classes of quotient theories of $\theory$ and the subtoposes of $\topos_\theory$.
\end{coro}
\begin{proof}
	If $\theory'$ is a quotient theory of $\theory$, then the map
	\begin{align*}
		{\mathfrak{e}^{\theory'}_\theory}_{\vec{x}}^{-1} \colon F^\theory(\vec{x}) & \to F^{\theory'}(\vec{x}) , \\
		 \form{\varphi}{x}_\theory & \mapsto \form{\varphi}{x}_{\theory'}
	\end{align*}
	that sends the $\theory$-provable equivalence class of a formula $\varphi$ in context $\vec{x}$ to the $\theory'$-provable equivalence class of the same formula yields a surjective map $F^\theory(\vec{x}) \twoheadrightarrow F^{\theory'}(\vec{x})$.  Moreover, since ${\mathfrak{e}^{\theory'}_\theory}_{\vec{x}}^{-1}$ preserves the interpretation of the logical symbols $\left\{\,\top,\land,\bigvee\,\right\}$, the map ${\mathfrak{e}^{\theory'}_\theory}_{\vec{x}}^{-1}$ is a frame homomorphism.  Additionally, it is easily observed that ${\mathfrak{e}^{\theory'}_\theory}_{\vec{x}}^{-1}$ is natural with respect to the maps $F^\theory(\sigma)$ and $\exists_{F^\theory(\sigma)}$ since ${\mathfrak{e}^{\theory'}_\theory}_{\vec{x}}^{-1}$ preserves substitution and the interpretation of the logical symbols $\left\{\,=, \exists \,\right\}$.  Therefore, the maps ${\mathfrak{e}^{\theory'}_\theory}_{\vec{x}}^{-1}$ are components of an internal sublocale embedding ${\mathfrak{e}^{\theory'}_\theory} \colon F^{\theory'} \rightarrowtail F^\theory$ and thus $\theory'$ yields a subtopos 
	\[\Sh({\mathfrak{e}^{\theory'}_\theory}) \colon \Sh(F^{\theory'}) \rightarrowtail \Sh(F^\theory) \simeq \topos_\theory.\]
	
	For the converse, a subtopos $f \colon \ftopos \rightarrowtail  \Sh(F^\theory) \simeq \topos_\theory$ must be induced by an internal locale morphism, i.e. an internal locale morphism $\fb \colon \Lb \to F^\theory$ where, for each $\vec{x}$, the component frame homomorphism $\fb_{\vec{x}}^{-1} \colon F^\theory(\vec{x}) \to \Lb(\vec{x})$ is surjective.  Let $\theory_f$ be the quotient theory of $\theory$ whose axioms consist of the sequents
	\[
	\varphi \vdash_{\vec{x}} \psi 
	\]
	for each pair of formulae $\varphi, \psi$ for which $\fb_{\vec{x}}^{-1}(\form{\varphi}{x}_\theory) \leqslant \fb_{\vec{x}}^{-1}(\form{\psi}{x}_\theory)$.  To complete the proof, we need only note that $\theory_{\Sh\left({\mathfrak{e}^{\theory'}_\theory}\right)} \equiv_s \theory'$.
\end{proof}

%%%%%%%%%%%%%%%%%%%%%%%%%%%%%%%%%%%%%%%%%%%%%%

%%%%%%%%%%%%%%%%%%%%%%%%%%%%%%%%%%%%%%%%%%%%%%%%%

%%%%%%%%%%%%%%%%%%%%%%%%%%
%%%%%%%%%%%%%%%%%%%%%%%%%%%%
%%%%%%%%%%%%%%%%%%%%%%%%%%%%

\section{The frame of internal nuclei}\label{frameintnuc}

In this final section, we consider the poset of internal nuclei on an internal locale, and demonstrate that it forms a frame whose frame operations can be computed `pointwise'.

Let $L$ be a locale and let $N(L)$ denote the set of nuclei on $L$.  We can order $N(L)$ by setting $j \leqslant k$ if $j(U) \leqslant k(U)$ for all $U \in L$.  Recall, from \cite[Proposition II.2.5]{stone} say, that so ordered $N(L)$ is a frame.  The set of sublocales of $L$, written as $\Sub_\Loc(L)$, can also be ordered with $[K \rightarrowtail L] \leqslant [K' \rightarrowtail L]$ if and only if there is a factorisation
\[\begin{tikzcd}[column sep=tiny]
	K \ar[tail]{rd}  \ar{rr} && K' \ar[tail]{ld} \\
	& L. &
\end{tikzcd}\]
Under the bijection between nuclei and sublocales, this is precisely the order dual $N(L) \cong \Sub_\Loc(L)\op$, and hence $\Sub_\Loc(L)$ is a co-frame.

\begin{dfs}
		Let $\Lb \colon \cat\op \to \Frm_{\rm open}$ be an internal locale of $\Sh(\cat,J)$, and let $\topos$ be a topos.
		\begin{enumerate}
			\item By $N(\Lb)$ we denote the poset of internal nuclei on $\Lb$ ordered by $j \leqslant k$ if and only if, for each $c \in \cat$ and $U \in \Lb_c$, $j_c(U) \leqslant k_c(U)$.
			\item By ${\bf LT}(\topos)$ we denote the poset of Lawvere-Tierney topologies for $\topos$, ordered by $j \leqslant k$ if and only if $j= j \land k$ (this poset is denoted as ${\bf Lop}(\topos)$ in \cite[\S A4.5]{elephant}).
			\item By $\Sub_{\Topos}(\topos)$ we denote the poset of subtoposes of $\topos$ which have been ordered by $[\ftopos' \rightarrowtail \topos]\leqslant [\ftopos \rightarrowtail \topos]$ if and only if there is a factorisation of geometric morphisms
			\[\begin{tikzcd}[column sep = tiny]
				\ftopos' \ar[tail]{rd} \ar{rr} && \ftopos \ar[tail]{ld} \\
				& \topos.&
			\end{tikzcd}\]
			\item By $\Sub_{\Loc(\Sh(\cat,J))}(\Lb)$ we denote the poset of internal sublocales of $\Lb$ ordered by $[\Lb' \rightarrowtail \Lb] \leqslant [\Lb'' \rightarrowtail \Lb]$ if and only if there is a factorisation of internal locale morphisms
			\[\begin{tikzcd}[column sep = tiny]
				\Lb' \ar[tail]{rd} \ar{rr} && \Lb'' \ar[tail]{ld} \\
				& \Lb.&
			\end{tikzcd}\]
		\end{enumerate}
\end{dfs}

Under the bijections established in Theorem \ref{intsublocthm}, there is an isomorphism of posets:
\[N(\Lb) \cong {\bf LT}(\Sh(\Lb)) \cong \Sub_{\Topos}(\Sh(\Lb))\op  \cong \Sub_{\Loc(\Sh(\cat,J))}(\Lb)\op\]
(where the latter two posets are the order duals of $\Sub_{\Topos}(\Sh(\Lb))$ and $\Sub_{\Loc(\Sh(\cat,J))}(\Lb)$ respectively).  We know already that $\Sub_{\Topos}(\Sh(\Lb))$ is a complete co-Heyting algebra, i.e. a co-frame (see \cite[\S A4.5]{elephant}).  We will give an alternative proof using internal nuclei that $N(\Lb)$ is a frame.

Moreover, we will show that the frame operations of $N(\Lb)$ can be computed `pointwise'.  That is, for each subset $\{\,j^i \mid i \in I\,\} \subseteq N(\Lb)$ and each object $c$ of $\cat$, there are equalities
\[\left(\bigwedge_{i \in I} j^i\right)_c = \bigwedge_{i \in I} j^i_c, \ \ \left(\bigvee_{i \in I} j^i\right)_c = \bigvee_{i \in I} j^i_c,\]
where $\bigwedge_{i \in I} j^i_c$ and $\bigvee_{i \in I} j^i_c$ are the meets and joins as computed in $N(\Lb_c)$.  The first of these equalities is easily shown.

\begin{lem}\label{lem:meetnucl}
	The meet of a subset $\{\,j^i \mid i \in I\,\} \subseteq N(\Lb)$ is given by
	\begin{equation}\label{meetnucleq}
		\left(\bigwedge_{i \in I} j^i\right)_c(U) = \bigwedge_{i \in I} j^i_c(U),
	\end{equation}
	for each $c \in \cat$ and $U \in \Lb_c$.
\end{lem}
\begin{proof}
	If (\ref{meetnucleq}) defines a valid internal nucleus on $\Lb$, it must clearly be the meet of the subset $\{\,j^i \mid i \in I\,\} \subseteq N(\Lb)$.  Recall from \cite[Proposition II.2.5]{stone} that $\bigwedge_{i \in I} j^i_c$ yields a nucleus on $\Lb_c$, for each object $c \in \cat$.  We must show naturality.  As $g^{-1} \colon \Lb_d \to \Lb_c$ is open, for an arrow $c \xrightarrow{g} d$ of $\cat$, it preserves all meets and so
	\[g^{-1}\left( \bigwedge_{i \in I} j^i_d(U)\right) =  \bigwedge_{i \in I} g^{-1}j^i_d(U) =  \bigwedge_{i \in I} j^i_cg^{-1}(U).\]
	Thus, $\bigwedge_{i \in I} j^i$ defines an internal nucleus on $\Lb$. 
\end{proof}

We will demonstrate that $N(\Lb)$ is a frame by generalising the notion of a pre-nucleus on a locale, recalled below, to the internal setting.

We give some justification as to why the frame operations can be computed `pointwise' as described in Theorem \ref{NLisframe} below.  Recall that the subtoposes of $\Sh(\dcat,J)$ correspond to Grothendieck topologies $J'$ on $\dcat$ that contain $J$.  In the case of a Grothendieck topology $J'$ on $\cat \rtimes \Lb$ that contains $K_\Lb$, we observe that the added data is generated by new covering families on the fibres $\Lb_c$.  Specifically, adding a new covering family
\[\{\,(c_i,U_i) \xrightarrow{f_i} (c,U) \mid i \in I\,\}\]
to $K_\Lb$ is equivalent to requiring that the family \[\{\,(c,\exists_{f_i} U_i) \xrightarrow{\id_c} (c,U) \mid i \in I \,\}\]
is covering.

%%%%%%%%%%%%%%%%%%%%%%%%%%%%%%%%%%%
%%%%%%%%%%%%%%%%%%%%%%%%%%%%%%%%%
%%%%%%%%%%%%%%%%%%%%%%%%%%%%%%%%%%%%%%%

\subsection{Pre-nuclei of locales}

There are many proofs of the fact that $N(L)$ is a frame for each locale $L$.  For example, the proof found in \cite[Proposition II.2.5]{stone} shows that $N(L)$ is a complete Heyting algebra by defining the Heyting operation.  Alternative approaches using pre-nuclei are considered in \cite{simmons} and \cite{escardo}.  We will follow the argument of \cite{simmons} when developing our internal generalisation.  We briefly repeat the argument for locales below.

Recall from \cite[\S 2]{simmons} that a pre-nucleus on a locale $L$ is a (necessarily monotone) map $p \colon L \to L$ that is inflationary and finite-meet-preserving: that is, for all $U, V \in L$,
\[U \leqslant p(U) , \ \ p(U \land V) = p(U) \land p(V).\]
Thus, a nucleus on $L$ is simply an idempotent pre-nucleus.  Unlike nuclei, pre-nuclei are closed under composition.

We denote by $PN(L)$ the poset of pre-nuclei on $L$ ordered by $p \leqslant q$ if $p(U) \leqslant q(U)$ for all $U \in L$.  It is clear that $PN(L)$ is a complete lattice: for each subset $\{\, p^i \mid i \in I \,\} $ of $PN(L)$ and each $U \in L$,
\[\left(\bigwedge_{i \in I} p^i\right)(U) = \bigwedge_{i \in I} p^i(U), \ \ \left(\bigvee_{i \in I} p^i\right)(U) = \bigvee_{i \in I} p^i(U),\]
where $\bigwedge_{i \in I} p^i(U)$ and $\bigvee_{i \in I} p^i(U)$ are calculated as in $L$.  It follows by the infinite distributive law for $L$ that $PN(L)$ is also a frame.

The inclusion of nuclei into pre-nuclei $N(L) \hookrightarrow PN(L)$ has a left adjoint 
\[{(-)}^\infty \colon PN(L) \to N(L),\]
which we call the \emph{nucleation} (the \emph{nuclear reflection} in \cite{escardo} and \emph{idempotent closure} in \cite{simmons}), constructed as follows. For each ordinal $\alpha$ and limit ordinal $\lambda$, we define inductively:
\[p^0(U) = U, \ \ p^{\alpha + 1} (U) = p(p^\alpha(U)), \ \ p^\lambda(U) = \bigvee_{\alpha < \lambda} p^\alpha(U).\]
At each stage, the resultant map $p^\kappa \colon L \to L$ is a pre-nucleus.  Necessarily, as $L$ is small, there is a sufficiently large ordinal $\kappa$ such that $p^\kappa$ is idempotent and therefore a nucleus.  We label this by $p^\infty$.  We observe that if $p \leqslant q$ then $p^\infty \leqslant q^\infty$, that $p \leqslant p^\infty$, and if $j$ is a nucleus then $j = j^\infty$.  That is, nucleation is bifunctorial, and has units and counits yielding the adjunction
\[\begin{tikzcd}
	N(L) \ar[hook, ""{name=1}]{r} & PN(L) \ar[bend right, ""{name=0}]{l}['] {(-)^\infty} \ar[phantom, "\perp", from=0, to = 1]
\end{tikzcd}\]
witnessing $N(L)$ as a reflective subcategory of $PN(L)$.

Thus, $N(L)$, in addition to the meets constructed in Lemma \ref{lem:meetnucl}, has all joins: for a subset 
\[{\{\,j^i \mid i \in I \,\} \subseteq N(L)},\]
the join in $N(L)$ is given by $\left(\bigvee_{i \in I}j^i\right)^\infty$.  The infinite distributive law for $N(L)$, and hence the fact that $N(L)$ is a frame, is a consequence of Lemma \ref{infdistforN} below (the lemma is equivalent to \cite[Lemma 3.1]{simmons}).

\begin{lem}\label{infdistforN}
	Let $L$ be a locale, $n$ a nucleus on $L$, and let $\{\,p^i \mid i \in I\,\}$ be a collection of pre-nuclei on $L$.  The infinite distributive law
	\[\left(n \land \bigvee_{i \in I } p^i\right)^\infty = n \land \left( \bigvee_{i \in I} p^i\right)^\infty\]
	holds.
\end{lem}
\begin{proof}
	We will show that $\left(n \land \bigvee_{i \in I } p^i\right)^\kappa = n \land \left( \bigvee_{i \in I} p^i\right)^\kappa$, for each ordinal $\kappa$, and thereby deduce the result.  The base case
	\[\left(n \land \bigvee_{i \in I } p^i\right)^0 = \id_L = n \land \left( \bigvee_{i \in I} p^i\right)^0\]
	is trivial.
	
	Suppose that $\left(n \land \bigvee_{i \in I } p^i\right)^\alpha = n \land \left( \bigvee_{i \in I} p^i\right)^\alpha$, then:
	\[\begin{split}
		\left(n \land \bigvee_{i \in I } p^i\right)^{\alpha + 1} & = \left(n \land \bigvee_{i \in I } p^i\right) \left(n \land \bigvee_{i \in I } p^i\right)^\alpha, \\
		& = n\left(\left(n \land \bigvee_{i \in I } p^i\right)^\alpha\right) \land \bigvee_{i \in I}p^i \left(\left(n \land \bigvee_{i \in I } p^i\right)^\alpha\right), \\
		& = n \land n \left(\left( \bigvee_{i \in I} p^i \right)^\alpha \right) \land \bigvee_{i \in I} p^i n \land p^i \left(\left( \bigvee_{i \in I} p^i \right)^\alpha \right).
	\end{split}\]
	Using that $n \leqslant n \left(\left( \bigvee_{i \in I} p^i \right)^\alpha \right)$, and $n \leqslant p^i n$, for all $i$, we have that:
	\[\begin{split}
		\left(n \land \bigvee_{i \in I } p^i\right)^{\alpha + 1} & = n \land \bigvee_{i \in I} p^i n \land p^i \left(\left( \bigvee_{i \in I} p^i \right)^\alpha \right), \\
		& = \bigvee_{i \in I} n \land p^i n \land p^i \left(\left( \bigvee_{i \in I} p^i \right)^\alpha \right), \\
		& = \bigvee_{i \in I} n \land p^i \left(\left( \bigvee_{i \in I} p^i \right)^\alpha \right), \\
		& = n \land \left(\bigvee_{i \in I} p^i \right)^{\alpha+1}.
	\end{split}\]
	
	Finally, if $\lambda$ is a limit ordinal such that $\left(n \land \bigvee_{i \in I } p^i\right)^\alpha = n \land \left( \bigvee_{i \in I} p^i\right)^\alpha$ for each ordinal $\alpha < \lambda$, then:
	\[\begin{split}
		\left(n \land \bigvee_{i \in I } p^i\right)^\lambda &= \bigvee_{\alpha < \lambda} \left(n \land \bigvee_{i \in I } p^i\right)^\alpha, \\
		& = \bigvee_{\alpha < \lambda} n \land \left( \bigvee_{i \in I} p^i\right)^\alpha, \\
		& = n \land \left( \bigvee_{i \in I} p^i\right)^\lambda.
	\end{split}\]
\end{proof}

%%%%%%%%%%%%%%%%%%%%%%%%
%%%%%%%%%%%%%%%%%%%%%%%%

\subsection{Internal pre-nuclei}

We now extend the theory of pre-nuclei and nucleation to the internal context.  In doing so we will observe that $N(\Lb)$ is a frame for every internal locale.

\begin{df}
	Let $\Lb$ be an internal locale of $\Sh(\cat,J)$.  An \emph{internal pre-nucleus} is a natural transformation $p \colon \Lb \to \Lb$ such that, for each $c \in \cat$, $p_c \colon \Lb_c \to \Lb_c$ is a pre-nucleus.  The set of internal pre-nuclei, denoted by $PN(\Lb)$, can be ordered by $p \leqslant q$ if $p_c(U) \leqslant q_c(U)$ for all $c \in \cat$ and $U \in \Lb_c$.
\end{df}

The poset of internal pre-nuclei $PN(\Lb)$ on an internal locale $\Lb$ of $\Sh(\cat,J)$ has all meets and all joins, which are computed `pointwise'.  Thus, by the infinite distributivity law for $\Lb_c$, for each $c \in \cat$, $PN(\Lb)$ is a frame.  We show that an internal nucleation can also be performed `pointwise'.

\begin{lem}
	Let $p \colon \Lb \to \Lb$ be an internal pre-nucleus on an internal locale $\Lb$, fibred over a category $\cat$.  The pointwise nucleations $p^\infty_c \colon \Lb \to \Lb_c$ of each component $p_c$ of $p$ are the components of an internal nucleus.
\end{lem}
\begin{proof}
	For each object $c \in \cat$, the nucleation $p_c^\infty \colon \Lb_c \to \Lb_c$ of $p_c$ is a nucleus, so it remains only show that they are natural in $c$.  This is easily shown by induction.  We will perform the case for a limit ordinal $\lambda$.  Let $g \colon c \to d$ be an arrow of $\cat$.  If, for all $\alpha < \lambda$, the square
	\[\begin{tikzcd}
		\Lb_d \ar{r}{g^{-1}} \ar{d}{p_d^\alpha} & \Lb_c \ar{d}{p_c^\alpha} \\
		\Lb_d \ar{r}{g^{-1}} & \Lb_c
	\end{tikzcd}\]
	commutes, then we have the desired equality
	\[g^{-1}\left( \bigvee_{\alpha < \lambda} p_d^\alpha\right) = \bigvee_{\alpha < \lambda} g^{-1}p_d^\alpha = \bigvee_{\alpha < \lambda} p_c^\alpha g^{-1}.\]
\end{proof}

As a result, we obtain a left adjoint to the inclusion $N(\Lb) \hookrightarrow PN(\Lb)$,
\[\begin{tikzcd}
	N(\Lb) \ar[hook, ""{name=1}]{r} & PN(\Lb), \ar[bend right, ""{name=0}]{l}['] {(-)^\infty} \ar[phantom, "\perp", from=0, to = 1]
\end{tikzcd}\]
just as we did for locales.  The functor $(-)^\infty \colon PN(\Lb) \to N(\Lb)$, the \emph{internal nucleation}, sends internal pre-nuclei to their pointwise nucleation.

\begin{thm}\label{NLisframe}
	Let $\Lb$ be an internal locale of $\Sh(\cat,J)$.  The poset $N(\Lb)$ of internal nuclei is a frame whose frame operations can be computed `pointwise' in that, for each subset $\{\,j^i \mid i \in I\,\} \subseteq N(\Lb)$ and each object $c$ of $\cat$, there are equalities
	\begin{equation}\label{pontwisecompt}
		\left(\bigwedge_{i \in I} j^i\right)_c = \bigwedge_{i \in I} j^i_c, \ \ \left(\bigvee_{i \in I} j^i\right)_c = \bigvee_{i \in I} j^i_c,
	\end{equation}
	where $\bigwedge_{i \in I} j^i_c$ and $\bigvee_{i \in I} j^i_c$ are computed as in $N(\Lb_c)$.
\end{thm}
\begin{proof}
	We saw in Lemma \ref{intnucgivessubloc} that $N(\Lb)$ has all meets and that these are computed pointwise.  The join of $\{\,j^i \mid i \in I\,\} \subseteq N(\Lb)$ is the nucleation of the join of $\{\,j^i \mid i \in I\,\} $ as internal pre-nuclei.  Since the nucleation of internal pre-nuclei is computed pointwise, as are joins in $PN(\Lb)$, the joins in $N(\Lb)$ are also computed pointwise in the sense of (\ref{pontwisecompt}).  Finally, as $N(\Lb_c)$ satisfies the infinite distributivity law for each $c \in \cat$, we obtain the infinite distributivity law for $N(\Lb)$.
\end{proof}

Since every Grothendieck topos $\topos$ is the topos of sheaves $\Sh(\Lb)$ for some internal locale $\Lb$ (see, for example, \cite[Proposition VII.3.1]{JT}), and also because $\Sub(\topos) \cong N(\Lb)\op$, we have recovered the well-known fact that the poset of subtoposes of a Grothendieck topos is a co-frame.

\begin{rem}
		Let $\Lb \colon \cat\op \to \Frm_{\rm open}$ be an internal locale of $\Sh(\cat,J)$.  Since the frame operations of $N(\Lb)$ are computed `pointwise', for each object $c$ of $\cat$, the projection $\pi_c \colon N(\Lb) \to N(\Lb_c)$ that sends an internal nucleus $j \colon \Lb \to \Lb$ to its component $j_c \colon \Lb_c \to \Lb_c$ at $c$ preserves all joins and meets.  Therefore, $\pi_c \colon N(\Lb) \to N(\Lb_c)$ is an open frame homomorphism.
\end{rem}

%%%%%%%%%%%%%%%%%%%%%%%%%%%%%%%%%%

%%%%%%%%%%%%%%%%%%%%%%%%%%%%%%%%%%%%%%

%%%%%%%%%%%%%%%%%%%%%%%%%%%%%%
%%%%%%%%%%%%%%%%%%%%%%%%%%%%%%%

%%%%%%%%%%%%%%%%%%%%%%%%%%%%%%%%%%%%%%%%%%%%%%%%%%%%%%%%%%%%%%
%%%%%%%%%%%%%%%%%%%%%%%%%%%%%%%%%%%%%%%%%%%%%%%%%%%%%%%%%%%%%%

%%%%%%%%%%%%%%%%%%%%%%%%%%%%%%%%
%%%%%%%%%%%%%%%%%%%%%%%%%%%%%%5
%%%%%%%%%%%%%%%%%%%%%%%%%%%%%%%%%%

%%%%%%%%%%%%%%%%%%%%%%%%%%%%%%%%%%%%%%%%%%%%%%%%%%%%%

%%%%%%%%%%%%%%%%%%%%%%%%%%%%%%%%%%%%%%%%%%%%%%%%%%%%%%

%%%%%%%%%%%%%%%%%%%%%%%%%%%%%%%%%%%%%%%%%%%%%%%%%%%%%%%

\refs

\bibitem [Artin, Grothendieck \& Verdier, 1973]{SGA4-3} M. Artin, A. Grothendieck and J.L. Verdier, {\it Théorie des Topos et Cohomologie Etale des Schémas, Séminaire de Géométrie Algébrique du Bois-Marie 1963-1964 (SGA 4)}. Springer Berlin-Heidelberg, 1973.

\bibitem [Caramello, 2018]{TST} O. Caramello, {\it Theories, sites, toposes: relating and studying mathematical theories through topos-theoretic `bridges'}.  Oxford University Press, 2018.

\bibitem [Caramello, 2020]{denseness} O. Caramello, ``Denseness conditions, morphisms and equivalences of toposes'', 2020. arXiv:\href{https://arxiv.org/abs/1906.08737}{1906.08737}.

\bibitem [Caramello, 2022]{fibredsites} O. Caramello, ``Fibred sites and existential toposes'', 2023. arXiv:\href{https://arxiv.org/abs/2212.11693}{2212.11693}.

\bibitem [Caramello \& Zanfa, 2021]{relsites} O. Caramello and R. Zanfa, ``Relative topos theory via stacks'', 2021. arXiv:\href{https://arxiv.org/abs/2107.04417}{2107.04417}.

\bibitem [Davey \& Priestley, 1990]{order} B. A. Davey and H. A. Priestley, {\it Introduction to Lattices and Order}. Cambridge University Press, 1990.

\bibitem [Escardó, 2003]{escardo} M. Escardó, ``Joins in the frame of nuclei''. {\it Applied categorical cites}, vol. 11, pp. 117--124, 2003.

\bibitem [Giraud, 1972]{giraud} J. Giraud, ``Classifying topos'', in {\it Toposes, Algebraic Geometry and Logic}, F. W. Lawvere, Ed., Springer Berlin Heidelberg, 1972, pp. 43--56.

\bibitem [Johnstone, 1977]{topos} P. T. Johnstone, {\it Topos theory}. Academic Press, 1977.

\bibitem [Johnstone, 1980]{open} P. T. Johnstone, ``Open maps of toposes''. {\it Manuscripta mathematica}, vol. 31, pp. 217--248, 1980.

\bibitem [Johnstone, 1981]{fact1} P. T. Johnstone, ``Factorization theorems for geometric morphisms, I''. {\it Cahiers de Topologie et Géométrie Différentielle Catégoriques}, vol. 22, no. 1, pp. 3--17, 1981.

\bibitem [Johnstone, 1982]{stone} P. T. Johnstone, {\it Stone spaces}, ser. Cambridge Studies in Advanced Mathematics. Cambridge University Press, 1982.

\bibitem [Johnstone, 1983]{pointless} P. T. Johnstone, ``The point of pointless topology''. {\it Bulletin of the American Mathematical Society}, vol. 8, no. 1, pp. 41--53, 1983.

\bibitem [Johnstone, 2002]{elephant} P. T. Johnstone, {\it Sketches of an Elephant: A topos theoretic compendium, Vol. 1 and 2}. Oxford University Press, 2002.

\bibitem [Joyal \& Tierney, 1984]{JT} A. Joyal and M. Tierney, ``An extension of the Galois theory of Grothendieck''. {\it Memoirs of the American Mathematical Society}, vol. 51, 1984.

\bibitem [Kock \& Moerdijk, 1991]{kockmoer} A. Kock and I. Moerdijk, ``Presentations of étendues''. {\it  Cahiers de Topologie et Géométrie Différentielle Catégoriques}, vol. 32, no. 2, pp. 145--164, 1991.

\bibitem [Lawvere, 1969]{adjoint} F. W. Lawvere, ``Adjointness in foundations''. {\it Dialectica}, vol. 23, no. 3/4, pp. 281--296, 1969.

\bibitem [Mac Lane \& Moerdijk, 1994]{SGL} S. Mac Lane and I. Moerdijk, {\it Sheaves in Geometry and Logic: A First Introduction to Topos Theory}. Springer New York, 1994.

\bibitem [Picado \& Pultr, 2012]{picadopultr} J. Picado and A. Pultr, {\it Frames and Locales: Topology without points}. Springer Basel, 2012.

\bibitem [Seely, 1983]{seely} R. A. G. Seely, ``Hyperdoctrines, natural deduction and the Beck condition''. {\it Mathematical Logic Quarterly}, vol. 29, no. 10, pp. 505--542, 1983.

\bibitem [Simmons, 1989]{simmons} H. Simmons, ``Near-discreteness of modules and spaces as measured by Gabriel and Cantor''. {\it Journal of Pure and Applied Algebra}, vol. 56, no. 2, pp. 119--162, 1989.

\bibitem [Street, 1980]{streetfibr1} R. Street, ``Fibrations in bicategories''. {\it {Cahiers de Topologie et Géométrie Différentielle Catégoriques}}, vol. 21, no. 2, pp. 111--160, 1980.
	
\bibitem [Street, 1981]{streetfibr2} R. Street, ``Conspectus of variable categories''. {\it Journal of Pure and Applied Algebra}, vol. 21, pp. 307--338, 1981.

\bibitem [Wrigley, 2021]{myself-presentation} J. L. Wrigley, ``The logic and geometry of localic morphisms''.  Talk at {\it Toposes Online}, 2021.  Available [Online]: \url{https://aroundtoposes.com/wp-content/uploads/2021/07/WrigleySlidesToposesOnline.pdf}.

\endrefs

\end{document}